\documentclass[12pt,reqno]{amsart}
\usepackage{hyperref}
\usepackage{geometry}
\usepackage{enumitem}
\usepackage{amsmath}
\usepackage{amssymb}
\usepackage{amsthm}
\usepackage{amsrefs}
\usepackage{amsfonts}
\usepackage[dvipsnames]{xcolor}
\usepackage{mathtools}
\usepackage{stackengine}
\usepackage{verbatim}
\usepackage{tikz-cd} 
\usepackage{mathrsfs}

\makeatletter
\@namedef{subjclassname@2020}{%
  \textup{2020} MSC}
\makeatother

\hypersetup{
	colorlinks   = true, 
	urlcolor     = blue, 
	linkcolor    = purple, 
	citecolor   = blue 
}

\newcommand{\diam}{\text{diam}}

\newcommand{\SL}{\mathrm{SL}}
\newcommand{\sspan}{\mathrm{span}}

\newcommand{\M}{\mathrm{M}}

\newcommand{\X}{\mathcal{X}}

\newcommand{\R}{\mathbb{R}}

\newcommand{\e}{\epsilon}

\newcommand{\Z}{\mathbb{Z}}

\newcommand{\N}{\mathbb{N}}

\newcommand{\an}{{\alpha_{\mathbf{\eta}}} }

\newcommand{\bfe}{\mathbf{e}}

\newcommand{\Mat}{\M_{m \times n}(\R)}
\newcommand{\Prob}{\text{Prob}}
\newcommand{\supp}{\text{supp}}
\newcommand{\Sing}{\text{Sing}}

\newcommand{\mc}{\mathcal}
\newcommand{\Fcal}{\mc{F}}
\newcommand{\Kcal}{\mc{K}}
\newcommand{\Pcal}{\mc{P}}
\newcommand{\Hcal}{\mc{H}}

\newcommand{\te}{{\tilde{e}}}
\newcommand{\mur}{{\mu^{(r)}}}
\newcommand{\tPcal}{\widetilde{\Pcal}}
\newcommand{\bfn}{{\hat{\eta}}}

\newcommand{\EMassl}{\underline{\text{EMass}}}
\newcommand{\EMassu}{\overline{\text{EMass}}}
\newcommand{\Div}{\text{Div}}
\newcommand{\Diver}{\text{Divergent}}

\newcommand{\Divergent}{\widetilde{\text{Div}}}

\newcommand{\oEM}{\overline{\text{EM}}}
\newcommand{\uEM}{\underline{\text{EM}}}

\title{Dimension bounds for singular affine forms}

\begin{document}
\theoremstyle{plain}
\newtheorem{thm}{Theorem}[section]
\newtheorem{lem}[thm]{Lemma}
\newtheorem{prop}[thm]{Proposition}
\newtheorem{cor}[thm]{Corollary}
\newtheorem{question}{Question}
\newtheorem{con}{Conjecture}
\theoremstyle{definition}
\newtheorem{defn}[thm]{Definition}
\newtheorem{exm}[thm]{Example}
\newtheorem{nexm}[thm]{Non Example}
\newtheorem{prob}[thm]{Problem}

\theoremstyle{remark}
\newtheorem{rem}[thm]{Remark}

\author{Gaurav Aggarwal}
\address{\textbf{Gaurav Aggarwal} \\
School of Mathematics,
Tata Institute of Fundamental Research, Mumbai, India 400005}
\email{gaurav@math.tifr.res.in}

\date{}

\thanks{ G. Aggarwal gratefully acknowledge a grant from the Department of Atomic Energy, Government of India, under project $12-R\&D-TFR-5.01-0500$. }



\begin{abstract}  
In this paper, we establish upper bounds on the dimension of sets of singular-on-average and \(\omega\)-singular affine forms in singly metric settings, where either the matrix or the shift is fixed. These results partially address open questions posed by Das, Fishman, Simmons, and Urbański, as well as Kleinbock and Wadleigh. Furthermore, we extend our results to the generalized weighted setup and derive bounds for the intersection of these sets with a wide class of fractals. 
\end{abstract}

\maketitle

\tableofcontents

\section{Introduction}

Fix \( m, n \in \N \) and vectors \( a = (a_1, \ldots, a_m) \in \R^m \) and \( b = (b_1, \ldots, b_n) \in \R^n \) such that  
\[
a_1 \ge a_2 \ge \cdots \ge a_m > 0, \qquad 
b_1 \ge b_2 \ge \cdots \ge b_n > 0,
\]
and  
\[
a_1 + \cdots + a_m = 1, \qquad b_1 + \cdots + b_n = 1.
\]
Let \( d = m + n \).  

We define a \emph{quasi-norm} \( \|\cdot\|_a \) on \( \R^m \) by  
\[
\|x\|_a = \max_i |x_i|^{1/a_i}, \qquad x = (x_1, \ldots, x_m) \in \R^m.
\]
Similarly, define a quasi-norm \( \|\cdot\|_b \) on \( \R^n \) by  
\[
\|y\|_b = \max_j |y_j|^{1/b_j}, \qquad y = (y_1, \ldots, y_n) \in \R^n.
\]

For \( \theta \in \Mat \) and \( \xi \in \R^m \), the \emph{inhomogeneous uniform \((a,b)\)-exponent} of \((\theta, \xi)\), denoted by \( \hat{\omega}(\theta, \xi, a, b) \), is defined as the supremum of all real numbers \( \omega \) such that the inequalities  
\begin{align*}
 \|p+ \theta q + \xi\|_a &\le \frac{1}{T^{1+\omega}}, \\
 \|q\|_b &\le T,
\end{align*}
admit an integer solution \((p, q) \in \Z^m \times (\Z^n \setminus \{0\})\) for all sufficiently large \( T \).  

When \( \xi = 0 \), we simply write \( \hat{\omega}(\theta, a, b) := \hat{\omega}(\theta, 0, a, b) \).  
Finally, define  
\[
\Sing(a, b, \omega) := \{\, (\theta, \xi) \in \Mat \times \R^m : \hat{\omega}(\theta, \xi, a, b) \ge \omega \,\}.
\]

\begin{rem}
The concept of Diophantine exponents was originally introduced by Khintchine~\cite{Khintchine} and Jarn{\'{i}}k~\cite{Jarnik}. For further details, see also~\cites{BugLau,CGGMS,German}.
\end{rem}

\begin{rem}
We note that the usual irrationality exponent, as defined in~\cite{BugLau}, is given by 
\(\frac{n}{m}(1 + \hat{\omega}(\theta, a, b))\), 
whereas in~\cite{CGGMS} it is defined as 
\(1 + \hat{\omega}(\theta, a, b)\).
\end{rem}

We define the set of \emph{$(a,b)$-singular affine forms}, denoted by \(\Sing(a,b)\), as the set of all 
\((\theta, \xi) \in \Mat \times \R^m\) 
such that for every \(\e > 0\), there exists \(T_{\e} > 0\) with the property that for all \(T > T_{\e}\), one can find 
\((p, q) \in \Z^m \times (\Z^n \setminus \{0\})\) satisfying
\begin{align*}
    \|p + \theta q + \xi\|_{a} &\leq \frac{\e}{T}, \\
    \|q\|_{b} &\leq T.
\end{align*}

\begin{rem}
It is straightforward to verify that for all \(\omega > 0\),
\[
\Sing(a,b,\omega) \subset \Sing(a,b).
\]
\end{rem}

\vspace{0.3in}

The study of inhomogeneous Diophantine approximation has a rich history; see, for example, \cite{aggarwalghosh2024inhom} for a historical review. In recent years, significant progress has been made, particularly in the area of \emph{uniform} inhomogeneous Diophantine approximation. Kleinbock and Wadleigh \cite{KleinbockWadleigh} proved an inhomogeneous \(\psi\)-Dirichlet theorem and, in Section 7 of their paper, posed questions regarding the zero-one law and the Hausdorff dimension of \(\psi\)-Dirichlet improvable systems of affine forms in singly metric cases, i.e., when either \(\theta\) or \(\xi\) is fixed. Kim and Kim \cite{KimKim} partially addressed these questions by proving a zero-one law for the \(s\)-dimensional Hausdorff measure of the \emph{complement} of the set of \(\psi\)-Dirichlet improvable systems of affine forms, both in the doubly metric case and in the singly metric case for fixed \(\xi \in \mathbb{R}^m\). See also \cite{BakhtawarSimmons} for further results. 

In joint work with A. Ghosh \cite{aggarwalghosh2024inhom}, the authors partially answered the measure-theoretic question of Kleinbock and Wadleigh for a wide class of measures, including natural measures on self-similar fractals and manifolds, corresponding to \(\psi(t) = t^{-\omega}\). For \(m=n=1\), the dimension question was addressed by Kim and Liao \cite{LiaoKim}, in the case where the real number \(\theta\) is fixed. However, for \((m,n) \neq (1,1)\), the problem of dimension estimates for \(\psi\)-Dirichlet improvable systems of affine forms, even for \(\psi(t) = t^{-\omega}\) with \(\omega > 1\) in the singly metric case, remains open.

The interest in dimension estimates in inhomogeneous Diophantine approximation for singly metric cases was also raised by Das, Fishman, Simmons, and Urbański. In \S 5.8 of \cite{DFSU}, they noted: ``It would be of interest to investigate analogues of our results in the frameworks of inhomogeneous approximation,'' and further, ``It is also natural to study the inhomogeneous approximation frameworks where we fix one coordinate of the pair \((\theta, \xi)\) and let the other vary.'' 

In this paper, we address these questions by providing upper bounds on the dimensions of singular-on-average and \(\omega\)-singular affine forms in both singly metric cases. More precisely, for a fixed \(\xi \in \R^m\), we define
\begin{align*}
    \Sing^\xi(a,b) &= \{\theta \in \Mat : (\theta, \xi) \in \Sing(a,b)\}, \\
    \Sing^\xi(a,b, \omega) &= \{\theta \in \Mat : (\theta, \xi) \in \Sing(a,b,\omega)\},
\end{align*}
and for a fixed \(\theta \in \Mat\), we define
\begin{align*}
    \Sing_\theta(a,b) &= \{\xi \in \R^m : (\theta, \xi) \in \Sing(a,b)\}, \\
    \Sing_\theta(a,b, \omega) &= \{\xi \in \R^m : (\theta, \xi) \in \Sing(a,b,\omega)\}.
\end{align*}
Let \(\dim_H(\cdot)\) and \(\dim_P(\cdot)\) denote the Hausdorff and packing dimensions, respectively. 

\medskip

The following statements follow from the main results of this paper.

\begin{cor}
\label{cor intro 1}
For Lebesgue almost every \(\theta \in \Mat\), we have
\[
\dim_H(\Sing_\theta(a,b)) = \dim_P(\Sing_\theta(a,b)) = 0.
\]
\end{cor}

\begin{cor}
\label{cor intro 2}
For any \(\xi \in \R^m\), we have
\begin{align*}
\dim_H(\Sing^\xi(a,b)) &\leq \dim_P(\Sing^\xi(a,b))
    \leq mn - \frac{1}{a_1 + b_1} \min\{m a_m, n b_n\}, \\
     \dim_H(\Sing^\xi(a, b, \omega) ) & \leq \dim_P(\Sing^\xi(a, b, \omega) ) \leq mn - \frac{1}{a_1 + b_1} \left( \min\{ma_m, nb_n\} + \frac{m a_m b_n \omega}{a_m + b_n + a_m \omega} \right).
\end{align*}
In particular, when \(a_1 = \cdots = a_m = 1/m\) and \(b_1 = \cdots = b_n = 1/n\), it follows that
\begin{align*}
    \dim_H(\Sing^\xi(a,b)) &\leq \dim_P(\Sing^\xi(a,b))
    \leq mn - \frac{mn}{m+n}, \\
     \dim_H(\Sing^\xi(a, b, \omega) ) & \leq \dim_P(\Sing^\xi(a, b, \omega) ) \leq mn - \frac{mn}{m+n} \left( 1 + \frac{ m \omega}{m + n + n \omega} \right).
\end{align*}
\end{cor}

\begin{rem}
The non-emptiness of \(\Sing(a,b)\) in the equal-weight case follows from the classical theorem of Khintchine~\cite{Khintchineinhom}; see also~\cite[\S3.3]{MoshchevitinNeckrasov}.  
For lower bounds on \(\Sing^\xi(a,b)\), we refer to Schleischitz~\cite{schleischitz2022}, who established \emph{lower bounds} for the packing dimension of singular vectors lying on certain classes of fractals in the equal-weight case with \(n = 1\).  
For \(n > 1\), the uncountability and density of \(\Sing^\xi(a,b)\) follow from~\cite[Thm.~8.1]{KMWW24}, where Kleinbock, Moshchevitin, Warren, and Weiss announced this result; the complete proof is expected to appear in forthcoming joint work with Hong and Neckrasov.  
For results concerning \(\Sing_\theta(a,b)\), see the paper of Moshchevitin and Neckrasov~\cite{MoshchevitinNeckrasov}.
\end{rem}

\begin{rem}
    For \(\xi = 0\), the problem of computing dimesnion of $\Sing^\xi(a,b)$ has been extensively studied, particularly in the equal weight case, \(a = (1/m, \ldots, 1/m)\) and $b= (1/n, \ldots, 1/n)$. A landmark result in this direction was obtained by Y. Cheung \cite{Cheung}, who showed that the Hausdorff dimension of \(\Sing^0((1/2, 1/2), 1) \subset \mathbb{R}^2\) is \(4/3\). This result was subsequently generalized to \(\mathbb{R}^m\) by Cheung and Chevallier \cite{CheungChevallier}. A sharp upper bound for the broader set of \emph{singular on average} \(m \times n\) matrices was later established by Kadyrov, Kleinbock, Lindenstrauss, and Margulis in \cite{KKLM}, using techniques from homogeneous dynamics. The complementary lower bound was proven by Das, Fishman, Simmons, and Urbański \cite{DFSU}, employing methods from the parametric geometry of numbers. In particular, it is now known that
    $$
    \Sing^0\left(\left(\frac{1}{m}, \ldots, \frac{1}{m}\right), \left( \frac{1}{n}, \ldots, \frac{1}{n} \right) \right) = mn - \frac{mn}{m+n}.
    $$
    For additional results in this direction, see also \cite{Solan}. 
    
    In \cite{Khalilsing}, Khalil provided an upper bound on the Hausdorff dimension of singular vectors lying on self-similar fractals in \(\mathbb{R}^m\) that satisfy the open set condition. Shah and Yang \cite{ShahYang} obtained dimension bounds for certain singular vectors lying on affine subspaces. However, in the unequal weight setting, the literature remains sparse. For \((m, n) = (2, 1)\), the Hausdorff dimension of \((a, b)\)-singular vectors was computed by Liao, Shi, Solan, and Tamam in \cite{LSST}. In \cite{KimPark2024}, Kim and Park derived a lower bound for \((a, b)\)-singular vectors in the case \(n = 1\). In joint work with A. Ghosh \cite{aggarwalghoshsingular}, the author obtained an upper bound for \((a, b)\)-singular vectors in \(\M_{m \times n}(\mathbb{R})\) and on products of self-similar fractals in \(\mathbb{R}\) satisfying the open set condition.

    For general \(\xi \neq 0\), progress has been limited. The only known result in this setting is due to Schleischitz \cite{schleischitz2022} as mentioned above.
\end{rem}

\begin{rem}
    For $\xi = 0$, the set of $\omega$-singular matrices has also been previously studied. In the unweighted setting, $\omega$-singular matrices have been investigated by Bugeaud, Cheung, and Chevallier \cite{BugeaudCheungChevallier}, Das, Fishman, Simmons, and Urbański \cite{DFSU}, and Schleischitz \cite{schleischitz2022}. In the weighted setting, to the best of the authors' knowledge, the only relevant work is the joint work of the author with A. Ghosh \cite{aggarwalghoshsingular}, which also provides a detailed historical review of the results.
\end{rem}

\begin{rem}
Using~\cite{Shi20}, it is easy to see that $\Sing^\xi(a,b,\omega)$ has full Lebesgue measure for all $\omega \leq 0$. 
In particular, $\dim_H(\Sing^\xi(a,b,\omega)) = mn$ for all $\xi$ and $\omega \leq 0$. 
For the dimension of the complement of $\Sing^\xi(a,b,\omega)$ when $\omega < 0$, see~\cite{KimKim}.
\end{rem}

\vspace{0.3in}

Given Corollaries~\ref{cor intro 1} and~\ref{cor intro 2}, and motivated by Mahler’s question in~\cite[Section~2]{Mahler} concerning Diophantine approximation on fractals (in particular, on the middle-thirds Cantor set), it is natural to ask about the behaviour of $\Sing(a,b)$ and $\Sing(a,b,\omega)$ when intersected with fractal sets. We also provide upper bounds in this setting. For brevity, we state here the results only for the middle-thirds Cantor set and refer the reader to Section~\ref{sec: Main Results} for further details.

\begin{cor}
\label{cor intro 3}
Suppose that \(a_1 = \cdots = a_m = 1/m\) and \(b_1 = \cdots = b_n = 1/n\). 
Let $\mu$ denote the $(\log 2 / \log 3)$-dimensional Hausdorff measure restricted to the middle-thirds Cantor set. 
Then, for $\mu^{m \times n}$-almost every \(\theta \in \Mat\), we have
\[
\dim_H(\Sing_\theta(a,b)) = \dim_P(\Sing_\theta(a,b)) = 0.
\]
\end{cor}

\begin{cor}
    \label{cor intro 4}
    Let $\mathcal{C}_3$ denote the middle-thirds Cantor set, and assume that either $m=1$ or $n=1$. Then for any \(\xi \in \R^m\), we have
    \begin{align*}
    \dim_H(\Sing^\xi(a,b)) &\leq \dim_P(\Sing^\xi(a,b))
        \leq \frac{\log 2}{ \log 3} \left( mn - \frac{1}{a_1 + b_1} \min\{m a_m, n b_n\} \right), \\
        \dim_H(\Sing^\xi(a, b, \omega) ) & \leq \dim_P(\Sing^\xi(a, b, \omega) ) \leq \frac{\log 2}{ \log 3} \left( mn - \frac{1}{a_1 + b_1} \left( \min\{m a_m, n b_n\} + \frac{m a_m b_n \omega}{a_m + b_n + a_m \omega} \right) \right).
    \end{align*}
    In particular, when \(a_1 = \cdots = a_m = 1/m\) and \(b_1 = \cdots = b_n = 1/n\), it follows that
    \begin{align*}
        \dim_H(\Sing^\xi(a,b)) &\leq \dim_P(\Sing^\xi(a,b))
        \leq \frac{\log 2}{ \log 3} \left( mn - \frac{mn}{m+n} \right), \\
        \dim_H(\Sing^\xi(a, b, \omega) ) & \leq \dim_P(\Sing^\xi(a, b, \omega) ) \leq \frac{\log 2}{ \log 3} \left( mn - \frac{mn}{m+n} \left( 1 + \frac{ m \omega}{m + n + n \omega} \right) \right).
    \end{align*}
\end{cor}

\begin{rem}
    In Section~\ref{sec: Main Results}, we will further study $\Sing_\theta(a,b)$ when $\theta$ is sampled from measures supported on non-degenerate curves, affine hyperplanes, and broader classes of fractals; see Corollary~\ref{cor 1} and Remark~\ref{rem: wide measures}. We will also provide upper bounds on $\Sing_\theta(a,b)$ explicitly in terms of dynamical properties of $\theta$; see Corollary~\ref{main thm 2'}. Furthermore, we will establish upper bounds for $\Sing^\xi(a,b)$ when intersected with a wide class of fractals, and obtain uniform estimates independent of $\xi$; see Theorem~\ref{main thm 1}.
\end{rem}

\medskip

\subsection{Key Ideas of the Paper}

The proof of the upper bounds in the case of a fixed matrix~$\theta$ relies on the following simple yet fundamental observation.  

Let $\Lambda$ be a homogeneous lattice in $\R^d$ whose shortest non-zero vector has length at least~$\e$. Suppose $x_1, x_2 \in \R^d$ satisfy $\|x_1 - x_2\| < \e/2$, where $\|\cdot\|$ denotes the supremum norm on
$\R^d$. Assume that both affine lattices $\Lambda + x_1$ and $\Lambda + x_2$ contain a vector of length less than $\delta < \e/4$. Then it follows that $\|x_1 - x_2\| < 2\delta$.  

The proof proceeds by iteratively applying this elementary observation. Specifically, it is used to construct a sequence of coverings of $\Sing_\theta(a,b)$, which improve at each step—fewer balls are required relative to their size—whenever the trajectory $(g_t u(\theta) \Z^d)_{t \geq 1}$ (see Section~\ref{sec: Main Results} for definitions) returns to a fixed compact set. While this iterative procedure is relatively simple in the real case, it becomes significantly more delicate when dealing with fractals. Using Lemma~\ref{lem: Falconer}, these coverings ultimately yield the desired dimension bounds stated in the theorem.

The proof of the upper bounds in the case of a fixed shift~$\xi$ is entirely different and closely follows the approach in~\cite{aggarwalghoshsingular}, which was itself inspired by~\cite{KKLM} and~\cite{Khalilsing}. The central idea is to construct a height function whose divergent trajectories correspond precisely to $\Sing^\xi(a,b)$. The construction of this height function is motivated by~\cite{Shi20}. Once the height function is established, the upper bound on the dimension follows directly from~\cite[Thm.~6.5]{aggarwalghoshsingular}.

However, a key difference between the present proof and that of~\cite{aggarwalghoshsingular} lies in the nature of the height function. The height function in~\cite{aggarwalghoshsingular} captures divergent trajectories under the diagonal flow, whereas the height function constructed here measures how long an orbit spends near a fixed homogeneous subspace—specifically, the space of homogeneous lattices within the space of affine unimodular lattices. This distinction is crucial: the new height function detects orbits that accumulate near smaller homogeneous subspaces, rather than focusing on divergence. Consequently, the present problem differs fundamentally from the study of divergent trajectories, and the associated height function is entirely different.

\medskip

\subsection{Structure of the Paper}  
Section~\ref{sec: Main Results} presents the main results of the paper. Thereafter, the paper is divided into two parts. The first part concerns the case where \(\theta\) is fixed and is devoted to the proof of Theorem~\ref{main thm 2}, while the second part treats the case where \(\xi\) is fixed and establishes Theorem~\ref{main thm 1}.  

The first part begins with Section~\ref{sec: Notation I}, which introduces the notation used throughout the paper. Section~\ref{sec: Dimension bound in Generalized Setup I} develops a generalized dynamical framework and establishes a version of Theorem~\ref{main thm 2} in that setting. Section~\ref{sec: Final Proof I} then completes the proof of Theorem~\ref{main thm 2}, concluding the first part.  

The second part begins with Section~\ref{sec: Notation II}, which introduces additional notation needed for the proof of Theorem~\ref{main thm 1}. Section~\ref{sec: Dimension Bound in Generalized Setup II} recalls relevant results from~\cite{aggarwalghoshsingular}. Section~\ref{sec: Height Function} is devoted to the construction of a height function, whose divergent trajectories correspond precisely to singular-on-average affine forms. Finally, Section~\ref{sec: Final Proof II} combines the ingredients from the preceding sections to complete the proof of Theorem~\ref{main thm 1}.

 \medskip

\subsection{Acknowledgements}
The author would like to thank Anish Ghosh for suggesting the problem
and for numerous discussions throughout the development of this paper.
The author is also grateful to Dmitry Kleinbock for useful comments
and for spotting some inaccuracies in an earlier draft.
The author thanks the anonymous referee for careful reading and
helpful comments which improved the clarity of the paper.

\section{Main Results}
\label{sec: Main Results}

\subsection{Homogeneous Spaces}
\label{subsec: Homogenous space}
We set
\[
\widetilde{G} = \SL_{d}(\R) \ltimes \R^{d}, 
\quad 
\widetilde{\Gamma} = \SL_{d}(\Z) \ltimes \Z^{d},
\]
where the group structure on $\widetilde{G}$ is given by
\[
[A, w][B, v] = [AB,\, w + A v].
\]
Also, denote by $\widetilde{\X} = \widetilde{G}/\widetilde{\Gamma}$ the corresponding finite-volume quotient.  
This quotient admits a natural interpretation as the space of \emph{affine unimodular lattices} in $\R^{d}$, that is, unimodular lattices accompanied by a translation.  
Explicitly, this correspondence is given by
\[
[A, v] \widetilde{\Gamma} \longmapsto A\Z^{d} + v.
\]

Similarly, we define
\[
G = \SL_{d}(\R), 
\quad 
\Gamma = \SL_{d}(\Z),
\]
and let \(\X = G / \Gamma\) denote the finite-volume quotient, naturally identified with the space of unimodular lattices in \(\R^{d}\) via
\[
A \Gamma \longmapsto A \Z^{d}.
\]
Throughout this paper, we regard \(G\) as a subgroup of \(\widetilde{G}\) via the embedding \(g \mapsto [g, 0]\), and similarly identify \(\X\) with a subset of \(\widetilde{\X}\) via \(A \Gamma \mapsto [A, 0] \widetilde{\Gamma}\).  

We denote by 
\[
\pi : \widetilde{\X} \to \X
\]
the natural projection map, explicitly given by
\[
\pi([A, v] \widetilde{\Gamma}) = A \Gamma.
\]

For \(t > 0\) and \(\theta \in \Mat\), define
\begin{equation}
\label{eq: def u x}
g_t = 
\begin{pmatrix}
t^{a_1} \\ & \ddots \\ && t^{a_m} \\ &&& t^{-b_1} \\ &&&& \ddots \\ &&&&& t^{-b_n}
\end{pmatrix},
\quad
u(\theta) =
\begin{pmatrix}
I_m & \theta \\ & I_n
\end{pmatrix}.
\end{equation}
For each \(\xi \in \R^m\), we define the vector \(v(\xi) \in \R^{d}\) as
\[
v(\xi) =
\begin{pmatrix}
\xi \\ 0
\end{pmatrix},
\]
that is, the first \(m\) coordinates of \(v(\xi)\) are given by \(\xi\), and the remaining \(n\) coordinates are zero.

\medskip

\subsection{Dani's correspondence}

By Dani’s correspondence, the Diophantine properties of a pair \((\theta, \xi)\) are reflected in the behavior of the diagonal orbit
\[
(g_t [u(\theta), v(\xi)] \widetilde{\Gamma})_{t \geq 1} \subset \widetilde{\X}.
\]
To state this correspondence precisely, let
\[
\lambda_0 : \widetilde{\X} \to [0,\infty), \quad 
\lambda_0(y) = \min\{ \|w\| : w \in y \setminus \{0\} \},
\]
where $\|\cdot\|$ denotes the supremum (i.e. $\ell^\infty$) norm on $\R^d$.

The following lemmas formalizes this connection. 

\begin{lem}
\label{lem: Sing Dynamical Interpretation}
If \((\theta, \xi) \in \Sing(a,b)\), then 
\[
\lambda_0(g_t [u(\theta), v(\xi)] \widetilde{\Gamma}) \longrightarrow 0
\quad \text{as } t \to \infty.
\]
\end{lem}
\begin{proof}
By definition, $(\theta,\xi) \in \Sing(a,b)$ if, for every $\delta > 0$, there exists $T_{\delta}>0$ such that for all $t > T_{\delta}$, there exist integers $(p, q) \in \Z^m \times (\Z^n \setminus \{0\})$ with the vector 
\[
z = (p + \theta q + \xi, q) \in u(\theta) \Z^d + v(\xi)
\] 
satisfying
\begin{align*}
    |z_i|^{1/a_i} &\leq \frac{\delta}{t}, \quad 1 \leq i \leq m, \\
    |z_{j+m}|^{1/b_j} &\leq t, \quad 1 \leq j \leq n.
\end{align*}

Set 
\[
\tau = \delta^{-a_m/(a_m+b_n)} t.
\] 
Then $g_{\tau}(u(\theta)\Z^d + v(\xi))$ contains the vector $g_{\tau} z = (z_1', \ldots, z_d')$, where
\begin{align*}
    |z_i'| &= \tau^{a_i} |z_i| \leq (\delta^{-a_m/(a_m+b_n)} t)^{a_i} (\delta t^{-1})^{a_i} = \delta^{a_m b_n/(a_m+b_n)}, \quad 1 \leq i \leq m, \\
    |z_{j+m}'| &= \tau^{-b_j} |z_{j+m}| \leq (\delta^{-a_m/(a_m+b_n)} t)^{-b_j} t^{b_j} = \delta^{a_m b_n/(a_m+b_n)}, \quad 1 \leq j \leq n,
\end{align*}
where we used that $a_m = \min_i a_i$ and $b_n = \min_j b_j$.  

This shows that 
\begin{align}
\label{eq: t 1}
\lambda_0(g_\tau [u(\theta), v(\xi)]\widetilde{\Gamma}) \leq \delta^{a_m b_n/(a_m+b_n)}.
\end{align}

Since \eqref{eq: t 1} holds for all $\tau > T_\delta \delta^{-a_m/(a_m+b_n)}$, and $\delta > 0$ is arbitrary, the lemma follows.
\end{proof}
\begin{rem}
\label{rem:interpret singular}
Note that for any sequence $(x_n)$ in $\widetilde{\X} \setminus \X$, we have $\lambda_0(x_n) \to 0$ if and only if the distance from $x_n$ to the homogeneous subspace $\X$ tends to $0$ as $n \to \infty$. In particular, the sequence $(x_n)$ may converge to a point in ${\X}$. Combined with Lemma~\ref{lem: Sing Dynamical Interpretation}, this shows that the study of singular affine forms is fundamentally different from the study of divergent trajectories in $\widetilde{\X}$.
\end{rem}

\medskip

\begin{lem}
\label{lem: omega sing dynamical Interpretation}
Let $0 < \omega < \omega'$. If $(\theta,\xi) \in \Sing(a,b,\omega')$, then there exists $T_{(\theta,\xi)} = T_{(\theta,\xi)}(\omega)$ such that for all $t > T_{(\theta,\xi)}$, we have
\[
\lambda_0(g_t [u(\theta), v(\xi)] \Z^d) \le t^{-\frac{a_m b_n \omega}{a_m + b_n + a_m \omega}}.
\]
\end{lem}

\begin{proof}
By definition, if $(\theta,\xi) \in \Sing(a,b,\omega')$, then there exists $T_{(\theta,\xi)}$ such that for all $t > T_{(\theta,\xi)}$, there exist integers $(p,q) \in \Z^m \times (\Z^n \setminus \{0\})$ with the vector
\[
z = (p + \theta q + \xi, q) \in u(\theta)\Z^d + v(\xi)
\]
satisfying
\begin{align*}
|z_i|^{1/a_i} &\le t^{-1-\omega}, \quad 1 \le i \le m,\\
|z_{j+m}|^{1/b_j} &\le t, \quad 1 \le j \le n.
\end{align*}

Set 
\[
\tau = t^{1 + \frac{a_m \omega}{a_m + b_n}}.
\]
Then $g_\tau (u(\theta)\Z^d + v(\xi))$ contains the vector $g_\tau z = (z_1', \dots, z_d')$, where
\begin{align*}
|z_i'| &= \tau^{a_i} |z_i| \le t^{a_i (1 + \frac{a_m \omega}{a_m + b_n})} t^{-a_i (1 + \omega)} \le t^{-\frac{a_m b_n \omega}{a_m + b_n}}, \quad 1 \le i \le m, \\
|z_{j+m}'| &= \tau^{-b_j} |z_{j+m}| \le t^{-b_j (1 + \frac{a_m \omega}{a_m + b_n})} t^{b_j} \le t^{-\frac{a_m b_n \omega}{a_m + b_n}}, \quad 1 \le j \le n,
\end{align*}
using that $a_m = \min_i a_i$ and $b_n = \min_j b_j$.  

Since $\tau = t^{1 + a_m \omega / (a_m + b_n)}$, this gives
\begin{align}
\label{eq: t 2}
    \lambda_0(g_\tau [u(\theta), v(\xi)] \Z^d) \le \tau^{-\frac{a_m b_n \omega}{a_m + b_n + a_m \omega}}.
\end{align}

Since \eqref{eq: t 2} holds for all $\tau > T_{(\theta,\xi)}^{1 + a_m \omega / (a_m + b_n)}$, the lemma follows.
\end{proof}

\medskip

\subsection{Singularity on average}

A more general way of quantifying the notion of singularity is through \emph{singularity on average}, introduced in \cite{KKLM} (see also \cite{DFSU}). More precisely, for $0 < q \leq 1$, we define $\Div(a,b,q)$ as set of all $(\theta, \xi) \in \Mat \times \R^m $ such that
\[
\lim_{\varepsilon \to 0} \liminf_{T \to \infty} \frac{1}{T} \, m_{\R}\big(\{ t \in [0,T] : \widetilde{\lambda}_0( [ g_{e^t}u(\theta), v(\xi)] \Z^d )  \leq \varepsilon \} \big) \geq q ,
\]
where $\widetilde{\lambda}_0: \widetilde{\X} \rightarrow [0, \infty)$ is defined as
$$\widetilde{\lambda}_0( y ) = \min\{\|v\| : v \in \Lambda\}.$$
\begin{rem}
    Note that $\widetilde{\lambda}_0(y) = 0$ if and only if $0 \in y$, i.e., $y$ belongs to the homogeneous subspace $\X$.
\end{rem}

For a fixed $\xi \in \R^m$, we define
\[
\Div^\xi(a,b,q) := \{\theta \in \Mat : (\theta, \xi) \in \Div(a,b,q)\},
\]
and for a fixed $\theta \in \Mat$, we define
\[
\Div_\theta(a,b,q) := \{\xi \in \R^m : (\theta, \xi) \in \Div(a,b,q)\}.
\]

The set $\Div(a,b,1)$ is often referred to as the set of $(a,b)$-singular-on-average affine forms, whereas $\Div^0(a,b,1)$ is referred to as the set of $(a,b)$-singular-on-average matrices.

In this paper we establish upper bounds for the Hausdorff and packing
dimensions of the sets $\Div_\theta(a,b,q)$ and $\Div^\xi(a,b,q)$.
Moreover, since by Lemma~\ref{lem: Sing Dynamical Interpretation} we have
\begin{align}
    \label{eq: inclusion}
    \Sing(a,b)\subset\Div(a,b,1),
\end{align}
these bounds immediately imply corresponding estimates for
$\Sing_\theta(a,b)$ and $\Sing^\xi(a,b)$.

\begin{rem}
As explained in Remark~\ref{rem:interpret singular}, the set $\Div(a,b,q)$—despite its name suggesting divergence—does not correspond to trajectories exhibiting at least $q$-escape of mass. Instead, it consists of trajectories that spend at least a $q$-proportion of time near the homogeneous subspace $\X \subset \widetilde{\X}$.
\end{rem}

\medskip

\subsection{Dimension bounds for fixed matrix $\theta$}

To state the results, we introduce the following notation. For $\theta \in \Mat$, define the lower and upper escape of mass of the trajectory $(g_t u(\theta)\Z^d)_{t \geq 1}$ as
\begin{align*}
\EMassl(\theta) &= \lim_{\varepsilon \to 0} \liminf_{T \to \infty} \frac{1}{T} m_\R\big(\{ t \in [0,T] : \lambda_0(g_{e^t} u(\theta)\Z^d) \leq \varepsilon \} \big), \\
\EMassu(\theta) &= \lim_{\varepsilon \to 0} \limsup_{T \to \infty} \frac{1}{T} m_\R\big(\{ t \in [0,T] : \lambda_0(g_{e^t} u(\theta)\Z^d) \leq \varepsilon \} \big).
\end{align*}

\begin{rem}
\label{rem: Meaning of EMass}
Using Mahler's compactness criterion, another way to interpret $\EMassl(\theta)$ and $\EMassu(\theta)$ is via weak-* limits of the measures
\begin{align}
\label{eq: xzs 1}
\frac{1}{T} \int_{0}^T \delta_{g_{e^t} u(\theta) \Gamma} \, dt
\end{align}
in $\X$.  

Specifically, $\EMassl(\theta)$ is the supremum of all $\alpha \geq 0$, and $\EMassu(\theta)$ is the infimum of all $\beta \geq 0$, such that every subsequential limit $\nu$ of the sequence \eqref{eq: xzs 1} satisfies
\[
\alpha \leq 1- \nu(\X) \leq \beta.
\]
\end{rem}

We now state our main result.

\begin{thm}
\label{main thm 2}
Let $r_1, \ldots, r_l \in \N$ satisfy $r_1 + \cdots + r_l = m$, and suppose that $a_i = a_j$ whenever there exists $1 \leq k \leq l$ such that 
\[
r_1 + \cdots + r_{k-1} < i \leq j \leq r_1 + \cdots + r_k,
\]
where $r_{0}:=0$.

For $1 \leq i \leq l$, let $w_i$ denote the common value of $a_j$ for indices $r_1 + \cdots + r_{i-1} < j \leq r_1 + \cdots + r_i$.

For each $1 \leq i \leq l$, let $\Phi_i$ be an iterated function system (IFS) of contracting similarities on $\R^{r_i}$ with equal contraction ratios, satisfying the open set condition (see Section~\ref{sec: Iterated Function Systems} for more details). Let $\Kcal_i$ denote the limit set of $\Phi_i$, and set
\[
s_i := \dim_H(\Kcal_i).
\]
Define
\[
\Kcal := \Kcal_1 \times \cdots \times \Kcal_l \subset \R^m.
\]

Then, for any $\theta \in \Mat$ and $0 < q \leq 1$, we have
\begin{align*}
\dim_H(\Div_\theta(a,b,q) \cap \Kcal) &\leq \min_{1 \leq k \leq l} \frac{1}{w_k} \sum_{i=1}^l s_i \Big( \max\{w_i, w_k\} - w_i q + w_i \EMassl(\theta) \Big), \\
\dim_P(\Div_\theta(a,b,q) \cap \Kcal) &\leq \min_{1 \leq k \leq l} \frac{1}{w_k} \sum_{i=1}^l s_i \Big( \max\{w_i, w_k\} - w_i q + w_i \EMassu(\theta) \Big).
\end{align*}
\end{thm}
\begin{rem}
    Note that in Theorem \ref{main thm 2}, we do not require $\dim_H(\Kcal_i) > 0$ for any $i$. However, if $\dim_H(\Kcal_i) = 0$ for all $i$, the bounds in the theorem are trivial. Furthermore, we do not require $r_1, \ldots, r_l$ to be maximal. For example, in the case of equal weights, that is, $a_1 = \cdots = a_m$, one could choose $l = m$ with $r_1 = \cdots = r_m = 1$, or $l = 1$ with $r_1 = m$, or any intermediate choice. 

    This flexibility, particularly in the equal weight case, allows us to study the sets $\Div_\theta(a, b, q) \cap \Kcal$ when $\Kcal$ is a single fractal (such as a line, a plane, or the Sierpiński carpet), or a product of fractals (such as the product of a middle-third Cantor set with a middle-fifth Cantor set).
\end{rem}

\medskip
We now state some immediate corollaries of Theorem~\ref{main thm 2}.

\begin{cor}
\label{main thm 2'}
For any $\theta \in \Mat$ and $0 < q \leq 1$, we have
\begin{align*}
\dim_H(\Div_\theta(a,b,q)) &\leq \min_{1 \leq k \leq m} \frac{1}{a_k} \sum_{i=1}^m \Big( \max\{a_i, a_k\} - a_i q + a_i \EMassl(\theta) \Big), \\
\dim_P(\Div_\theta(a,b,q)) &\leq \min_{1 \leq k \leq m} \frac{1}{a_k} \sum_{i=1}^m \Big( \max\{a_i, a_k\} - a_i q + a_i \EMassu(\theta) \Big).
\end{align*}
\end{cor}

\begin{cor}
\label{main thm 2''}
   Suppose \(a_1 = \cdots = a_m = 1/m\) and \(b_1 = \cdots = b_n = 1/n\). Then for any $\theta \in \Mat$ and $0 < q \leq 1$, we have
   \begin{align*}
\dim_H(\Div_\theta(a,b,q)) &\leq m(1-q) +  m\EMassl(\theta) , \\
\dim_P(\Div_\theta(a,b,q)) &\leq m(1-q) +  m\EMassu(\theta).
\end{align*}
In particular, we have
\begin{align*}
    \dim_H(\Sing_\theta(a,b)) &\leq  m\EMassl(\theta) , \\
\dim_P(\Sing_\theta(a,b)) &\leq  m\EMassu(\theta).
\end{align*}
\end{cor}

\begin{cor}
\label{cor 1}
For all $\theta \in \Mat$ satisfying $\EMassu(\theta) = 0$, we have
\[
\dim_H(\Sing_\theta(a,b)) = \dim_P(\Sing_\theta(a,b)) = 0.
\]
\end{cor}

\begin{cor}
\label{cor 2}
For all $\theta \in \Mat$ that are not $(a,b)$-singular-on-average matrices, i.e., $\theta \notin \Div^0(a,b,1)$, the set $\Sing_\theta(a,b)$ has zero Lebesgue measure. Moreover, its Hausdorff dimension is strictly less than $m$.
\end{cor}

\begin{rem}
\label{rem: wide measures}
Note that for equal weights, i.e., $a_1 = \cdots = a_m = 1/m$ and $b_1 = \cdots = b_n = 1/n$, there is a wide class of measures, beyond the Lebesgue measure, for which the condition $\EMassu(\theta) = 0$ holds almost everywhere. Consequently, these measures give full measure to the set of non-$(a,b)$-singular-on-average matrices. Examples include natural measures on limit sets of IFSs as in \cite{SimmonsWeiss} (e.g., the middle-third Cantor set or the Koch snowflake), natural measures on non-degenerate curves in $\M_{m \times n}(\R)$ as in \cite{solanwieser}, and natural measures on various affine subspaces under certain conditions, as in \cite{ShahYang}. For a discussion of the unequal weight case, see \cite{ProhaskaSertShi}.
\end{rem}

\begin{rem}
For further results analogous to Corollary~\ref{cor 1}, we refer the reader to \cite[Theorem~3.8 and Corollary~3.9]{MoshchevitinNeckrasov}.
\end{rem}

\medskip

\subsection{Dimension bounds for fixed shift $\xi$}
We now explore the other singly metric case, where the shift \(\xi\) is fixed, and the matrix \(\theta\) varies. In this case, our main result is the following.

   \begin{thm}
    \label{main thm 1}
    Assume the notations as above. For $1 \leq l \leq d-1$, define $w_l$ as follows:
    $$
    w_l = 
    \begin{cases} 
        a_m + \cdots + a_{m+1-l}, & \text{if } l \leq m, \\
       b_n+ \cdots+ b_{l-m+1}, & \text{if } l \geq m. 
    \end{cases}
    $$

    For each $1 \leq i \leq m$ and $1 \leq j \leq n$, let $\Phi_{ij}$ be an iterated function system (IFS) consisting of contracting similarities on $\R$ with equal contraction ratios, satisfying the open set condition (see Section~\ref{sec: Iterated Function Systems} for more details). Let $\Kcal_{ij}$ be the limit set of $\Phi_{ij}$, and define 
    $$
    \Kcal = \{\theta \in \Mat : \theta_{ij} \in \Kcal_{ij} \text{ for all } i, j\}.
    $$

    Assume that $\dim_H(\Kcal_{ij})>0$ for all $i, j$. Then there exist constants $\eta_1, \ldots, \eta_{d-1} > 0$ (depending only on $\Kcal$) such that the following results hold for any $\xi \in \R^m$ :
    \begin{itemize}
        \item For any $0<q \leq 1$, the packing dimension of $\Div^\xi(a,b,q) \cap \Kcal$ satisfies:
        \begin{align}
        \label{eq: main thm 1}
          \dim_P(\Div^\xi(a,b,q) \cap \Kcal) \leq \dim_P(\Kcal) - \frac{q}{a_1 + b_1} \left( \min_{1 \leq l \leq d-1} \eta_l w_l \right).
        \end{align}
        \item For any $\omega > 0$, the packing dimension of $\Sing^\xi(a, b, \omega) \cap \Kcal$ satisfies:
        \begin{align}
        \label{eq: main thm 2}
        \dim_P(\Sing^\xi(a, b, \omega) \cap \Kcal) \leq \dim_P(\Kcal) - \frac{1}{a_1 + b_1} \left( \min_{1 \leq l \leq d-1} \eta_l w_l + \frac{\eta_1 a_m b_n \omega}{a_m + b_n + a_m \omega} \right).
        \end{align}
    \end{itemize}

    Moreover, the constants $\eta_1, \ldots, \eta_{d-1}$ can be explicitly chosen in the following cases:
    \begin{enumerate}
        \item If $\Kcal = \M_{m \times n}([0,1])$, we can take:
        \begin{align}
        \label{eq: main thm 3}
        \eta_l = 
        \begin{cases}
            \frac{m}{l}, & \text{if } l \leq m, \\
            \frac{n}{m+n-l}, & \text{if } l > m.
        \end{cases}
        \end{align}

        \item If $n = 1$, we can take:
       \begin{align}
        \label{eq: main thm 4}
        \eta_l = \frac{m}{l} \min_{1 \leq i \leq m} \dim_H(\Kcal_{i1}).
       \end{align}

        \item If $m = 1$, we can take:
        \begin{align}
        \label{eq: main thm 5}
        \eta_l = \frac{n}{1+n-l} \min_{1 \leq i \leq m} \dim_H(\Kcal_{i1}).
       \end{align}
    \end{enumerate}
\end{thm}

\begin{rem}
\label{rem: Enough for irrational}
Since Theorem~\ref{main thm 1} is already established for $\xi \in \Z^m$ in \cite{aggarwalghoshsingular}, we will assume throughout the proof that $\xi \in \R^m \setminus \Z^m$.
\end{rem}

\medskip


\section{Notation I}
\label{sec: Notation I}
The following notation will be used throughout the paper.

\subsection{Hausdorff and Packing Dimensions}
The $i$-dimensional Hausdorff measure of a set $F \subset \R^l$ is defined as
\begin{align*}
    \Hcal^i(F)= \sup_{\e>0} \inf\left\{\sum_{j=1}^\infty (\diam(U_j))^i: \substack{(U_j)_{j=1}^\infty \text{ is a countable cover of } F \\ \text{ with $\diam(U_j) \leq \e$ for all $j$. }}  \right\}
\end{align*}
The Hausdorff dimension of a set $F \subset \R^l$ is defined as
\begin{align*}
    \dim_H(F)= \inf\{i: \Hcal^i(F)=0\}= \sup\{i:  \Hcal^i(F)= \infty\}.
\end{align*}
The $i$-dimensional packing measure of a set $F \subset \R^l$ is defined as 
\begin{align*}
    \Pcal^i(F) := \inf \left\{\sum_{j= 1}^\infty \tPcal^i(F_j):  F \subset \bigcup_{j=1}^\infty  F_j \right\}, 
\end{align*}
where 
\begin{align*}
    \tPcal^i(F) = \inf_{\e>0} \sup \left\{ \sum_{j=1}^\infty (\diam(B_j))^i : \substack{(B_j)_{j=1}^\infty  \text{ is a countable collection of disjoint balls } \\ \text{ with centers in $F$ and with $\diam(B_j) < \e$ for all $j$.}} \right\}
\end{align*}
The packing dimension of a set $F \subset \R^l$ is defined as
\begin{align*}
    \dim_P(F)= \inf\{i: \Pcal^i(F)=0\}= \sup\{i:  \Pcal^i(F)= \infty\}.
\end{align*}

We will need the following important lemma.
\begin{lem}[{\cite[Prop.~3.4 and Lem.~3.8]{Falconer}}]
\label{lem: Falconer}
    For $F$ a non-empty bounded subset of $\R^l$, we have
    \begin{align*}
        \dim_H(F) &\leq \liminf_{\delta \rightarrow 0} \frac{\log C_{\delta}(F)}{-\log \delta}, \\
         \dim_P(F) &\leq \limsup_{\delta \rightarrow 0} \frac{\log C_{\delta}(F)}{-\log \delta},
    \end{align*}
    where $C_\delta(F)$ denotes the smallest number of sets of diameter at most $\delta$ that cover $F$. In particular,
    $$
    \dim_H(F) \leq \dim_P(F).
    $$
\end{lem}

\subsection{Iterated Function Systems}
\label{sec: Iterated Function Systems}
    A contracting similarity is a map $\R^l \rightarrow \R^l$ of the form $x \mapsto c O x+ y$ where $c \in (0,1)$, $y \in \R^l$ and $O$ is a $l \times l$ special orthogonal matrix. A \emph{finite similarity Iterated Function System} with constant ratio ({\em IFS}) on $\R^l$ is a collection of contracting similarities $\Phi= (\phi_e: \R^l \rightarrow \R^l)_{e \in E}$ indexed by a finite set $E$, called the alphabet, such that there exists a constant $c \in (0,1)$ independent of $e$ so that
$$
\phi_e(x)= cO_e x + w_e,
$$
for all $e \in E$.

Let $B = E^\N$. The coding map of an IFS $\Phi$ is the map $\sigma: B \rightarrow \R^l$ defined by the formula
\begin{align}
\label{eq:def sigma}
    \sigma(b)= \lim_{j \rightarrow \infty} \phi_{b_1} \circ \cdots \circ \phi_{b_j}(0).
\end{align}
It is well known that the limit in $\eqref{eq:def sigma}$ exists and that the coding map is continuous. The image of $B$ under the coding map called the limit set of $\Phi$, is a compact subset of $\R^l$, which we denote by $\Kcal= \Kcal(\Phi)$. For details, we refer the reader to \cite[Chap.~9]{Falconer}.

We define for $\te=(e_1, \ldots, e_j) \in E^j$, 
\begin{align}
    \label{eq: def K w, collection}
    \Kcal_\te= \phi_{e_1} \circ \cdots \circ \phi_{e_j}(\Kcal), \quad \text{ and set } \Fcal(j)= \{\Kcal_{\te} : \te \in E^j\}. 
\end{align}

We will say that $\Phi$ satisfies the \emph{open set condition} (OSC for short) if there exists a non-empty open subset  $U \subset \R^l$ such that the following holds
    \begin{align*}
        \phi_e(U) \subset U &\text{ for every } e \in E \\
        \phi_e(U) \cap \phi_{e'}(U) = \emptyset, &\text{ for every } e \neq e' \in E.
    \end{align*}

Let $\Prob(E)$ denote the space of probability measures on $E$. For each $\nu \in \Prob(E)$ we can consider the measure $\sigma_* \nu^{\otimes \N}$ under the coding map. A measure of the form $\sigma_* \nu^{\otimes \N}$ is called a \emph{Bernoulli measure}. 

The following proposition is well known (see, for e.g. \cite[Prop.~5.1(4), Thm.~5.3(1)]{Hutchinson},  for a proof).

\begin{prop}
    \label{prop: important IFS}
    Suppose $\Phi= \{\phi_e: e \in E\}$ is an IFS satisfying the open set condition with the limit set $\Kcal$. Let $c$ denote the common contraction ratio of $(\phi_e)_{e \in E}$ and $p= \#E$. Then the Hausdorff and packing dimension of $\Kcal$ both equal $s:=-\log p/ \log c$. Also, the $s$-dimensional Hausdorff measure $\Hcal^s$ satisfies $0< \Hcal^s(\Kcal)< \infty$. Moreover if $\mu$ denotes the normalised restriction of $\Hcal^s$ to $\Kcal$, then $\mu$ is a Bernoulli measure and equals $\sigma_* \nu^{\otimes \N}$, where $\nu$ is the uniform measure on $E$, i.e, $\nu(F)= \#F/\#E$ for all $F \subset E$. Additionally, for every $j \in \N$ and distinct sequences $\te_1 \neq \te_2 \in E^j$, we have $\mu(\Kcal_{\te_1} \cap \Kcal_{\te_2}) = 0$. Furthermore, there exists a constant $\lambda >0$ such that for all $x \in \R^l$ and $y>0$, we have
    \begin{align}
        \label{eq: imp IFS}
        \mu(B(x,y)) \leq \lambda y^s.
    \end{align}
\end{prop}

We will also need the following lemma.
\begin{lem}
\label{lem: finite intersect}
    Suppose \(\Phi = \{\phi_e: e \in E\}\) is an IFS satisfying the open set condition with limit set \(\Kcal \subset \mathbb{R}^l\). Assume that \(\dim_H(\Kcal) > 0\). Let $\alpha$ denote the diameter of $\Kcal$ and let \(c\) denote the common contraction ratio of \((\phi_e)_{e \in E}\). Then there exists \(L \in \mathbb{N}\) such that the following holds: For every ball \(B \subset \mathbb{R}^l\) of radius \(\beta > 0\), there exist at most \(L\) elements in \(\Fcal(k_\beta)\) that intersect \(B\), where \(k_\beta\) is the unique integer such that 
    \begin{align}
        \label{eq: lem finite intersect}
        c^{k_{\beta}+1} \alpha < \beta \leq c^{k_{\beta}} \alpha,
    \end{align}
    and \(\Fcal(\cdot)\) is defined as in \eqref{eq: def K w, collection}. 
\end{lem}

\begin{proof}
    Suppose the Hausdorff dimension of \(\Kcal\) equals \(s\) and \(\mu\) is the normalized restriction of \(\Hcal^s\) to \(\Kcal\). Let \(\lambda > 0\) be such that \eqref{eq: imp IFS} holds for \(\mu\). Fix \(L \in \mathbb{N}\) such that \(L > \lambda (2c^{-1} \alpha )^s\). Suppose \(B\) is a ball of radius \(\beta >0\) and center \(x\). Let \(k_\beta\) be defined as above.  Define \(\tilde{B}\) as a ball of size \(2c^{-1} \beta\) with center \(x\). Then, for all \(R \in \Fcal(k_\beta)\) such that \(R \cap B \neq \emptyset\), we have \(R \subset \tilde{B}\). To see this, let \(y \in R \cap B\) be arbitrary. Then, for all \(z \in R\), we have 
    \begin{align*}
        \|x - z\| &\leq \|x - y\| + \|y - z\| \leq \beta + c^{k_{\beta}} \alpha \\
        &\leq \beta + c^{-1} \beta < 2c^{-1} \beta,
    \end{align*}
    which immediately implies \(R \subset \tilde{B}\). Thus we have
    \begin{align*}
        \#\{R \in \Fcal(k_\beta): R \cap B \neq \emptyset\} &\leq \frac{1}{c^{k_\beta s}} \sum_{\substack{R \in \Fcal(k_\beta) \\ R \cap B \neq \emptyset}} \mu(R) \leq \frac{1}{c^{k_\beta s}} \mu(\tilde{B}) \\
        &\leq \frac{1}{c^{k_\beta s}} \lambda (2c^{-1} \beta )^s \quad \text{ using \eqref{eq: imp IFS}}\\
        &\leq \lambda (2c^{-1} \alpha)^s \quad \text{ using \eqref{eq: lem finite intersect}} \\
        &< L.
    \end{align*}
    This proves the claim.
\end{proof}


\section{Dimension bound in Generalized Setup I}
\label{sec: Dimension bound in Generalized Setup I}
\begin{defn}
   With a slight abuse of notation, for $x \in \widetilde{\X}$, we define 
   \begin{align*}
       \EMassl(x) &= \lim_{\e \rightarrow 0} \liminf_{T \rightarrow \infty} \frac{1}{T} m_{\R} (\{ t \in [0,T]: \lambda_0(\pi(g_{e^T} x)) \leq \e   \}) , \\
       \EMassu(x) &= \lim_{\e \rightarrow 0} \limsup_{T \rightarrow \infty} \frac{1}{T} m_{\R} (\{ t \in [0,T]: \lambda_0(\pi(g_{e^T} x)) \leq \e   \}).
   \end{align*}
   Recall that $\lambda_0(y)$ denotes the length of the shortest non-zero vector in $y$. The significance of $\EMassl(x)$ and $\EMassu(x)$ is that any subsequential limit of measures $\frac{1}{T} \int_0^T \delta_{g_{e^t}u(\theta) x}$, say $\mu_x$ satisfies $1-\EMassu(x) \leq \mu_x(\widetilde{\X}) \leq 1-\EMassl(x)$.
\end{defn}

\begin{lem}
\label{lem: Emass}
    For any $x \in \Tilde{X}$, let us define for $t>1$ and $\e>0$, the sets
     \begin{align*}
        I(t, \e)&= \{k \in \N: \lambda_0(\pi(g_t^kx)) \leq \e\} , \\
        I(t, \e, N) &= I(t, \e) \cap [1, N].
    \end{align*}
    Then for any $t>1$, we have
    \begin{align*}
        \lim_{\e \rightarrow 0} \liminf_{N \rightarrow \infty} \frac{1}{N} \#I(t, \e, N) &= \EMassl(x) \\
        \lim_{\e \rightarrow 0} \limsup_{N \rightarrow \infty} \frac{1}{N} \#I(t, \e, N) &= \EMassu(x) \\
    \end{align*}
\end{lem}
\begin{proof}
    Fix $t>1$. For $\e>0$, let us define 
    \begin{align*}
        K_\e&= \{y \in \widetilde{\X}: \lambda_0(\pi(y)) \leq \e\}, \\
        {K}_{\e}' &= \bigcap_{s=t^{-1}}^t g_sK_\e \supset K_{\e t^{-1}} ,\\
        K_\e ''& = \bigcup_{s=t^{-1}}^t g_sK_\e \subset K_{\e t}.
    \end{align*}
    
    For $T>0$, define $N_T \in \N$ such that $t^{N_T} \leq e^{T}< t^{N_T+1}$. Then for any $\e>0$, it is easy to see that
    \begin{align}
        \frac{1}{T} \int_0^{T} \delta_{g_{e^s}x}(K_\e) \, ds \nonumber &= \frac{1}{T} \left( \sum_{j=1}^{N_T} \int_{-\log t}^0 \delta_{g_{t}^j g_{e^s} x} (K_\e) \, ds + \int_{N_T\log t}^{T} \delta_{ g_{e^s} x} (K_\e) \, ds \right) \nonumber \\
        &\leq \frac{1}{N_T\log t} \left(   \sum_{j=1}^{N_T}  \delta_{g_{t}^j  x} (K_\e'') \log t   + (T- N_T\log t) \right) \nonumber \\
        &\leq \frac{1}{N_T }\left( \#I(t, t{\e}, N_T) + 1 \right), \label{eq: abbc 1}
    \end{align}
    and
    \begin{align}
        \frac{1}{T} \int_0^{T} \delta_{g_{e^s}x}(K_\e) \, ds \nonumber &= \frac{1}{T} \left( \sum_{j=1}^{N_T} \int_{-\log t}^0 \delta_{g_{t}^j g_{e^s} x} (K_\e) \, ds + \int_{N_T\log t}^{T} \delta_{ g_{e^s} x} (K_\e) \, ds \right) \nonumber \\
        &\geq \frac{1}{(N_T+1)\log t} \left( \sum_{j=1}^{N_T}  \delta_{g_{t}^j  x} (K_\e') \log t  \right) \nonumber \\
        &\geq \frac{1}{N_T+1 }\left( \#I(t, \e t^{-1}, N_T)  \right). \label{eq: abbc 2}
    \end{align}
    From equations \eqref{eq: abbc 1} and \eqref{eq: abbc 2}, we get that
    \begin{align}
         \liminf_{T \rightarrow \infty}\frac{1}{N_T+1 }\left( \#I(t, \e t^{-1}, N_T) \right) \leq  \liminf_{T \rightarrow \infty} \frac{1}{T} \int_0^{T} \delta_{g_{e^s}x}(K_\e) \, ds \leq \liminf_{T \rightarrow \infty} \frac{1}{N_T }\left( \#I(t, t{\e}, N_T) + 1 \right), \label{eq: abbc 3}\\
         \limsup_{T \rightarrow \infty}\frac{1}{N_T+1 }\left( \#I(t, \e t^{-1}, N_T)  \right) \leq  \limsup_{T \rightarrow \infty} \frac{1}{T} \int_0^{T} \delta_{g_{e^s}x}(K_\e) \, ds \leq \limsup_{T \rightarrow \infty} \frac{1}{N_T }\left( \#I(t, t{\e}, N_T) + 1 \right). \label{eq: abbc 4}
    \end{align}
    The lemma now follows by taking limit as $\e \rightarrow 0$ in \eqref{eq: abbc 3} and \eqref{eq: abbc 4}, and noting that the sequence $\{N_T= \lfloor T/\log t \rfloor: T \geq \log t\}=\N$. Hence proved.
\end{proof}

\begin{defn}
    Given $x \in \X$ and $0 < q \leq 1 $, we define $\Divergent(x,q,a,b) \subset \R^m$ as the set of all $\xi \in \R^m$ such that 
    \begin{align*}
        \lim_{\e \rightarrow 0} \liminf_{T \rightarrow \infty} \frac{1}{T} m_{\R} ( \{ t \in [0,T]: \widetilde{\lambda}_0({g_{e^t}[I_d,v(\xi)] x}) \leq \e \} \geq q.\\
    \end{align*}
    Recall that $\widetilde{\lambda}_0(y)$ denotes the length of the shortest vector in the lattice $y$.
\end{defn}

\begin{thm}
\label{gen thm 2}
   Let $r_1, \ldots, r_l \in \N$ satisfy $r_1 + \cdots + r_l = m$, and suppose that $a_i = a_j$ whenever there exists $1 \leq k \leq l$ such that 
\[
r_1 + \cdots + r_{k-1} < i \leq j \leq r_1 + \cdots + r_k,
\]
where $r_0:=0$.

For $1 \leq i \leq l$, let $w_i$ denote the common value of $a_j$ for indices $r_1 + \cdots + r_{i-1} < j \leq r_1 + \cdots + r_i$.

For each $1 \leq i \leq l$, let $\Phi_i$ be an iterated function system (IFS) of contracting similarities on $\R^{r_i}$ with equal contraction ratios, satisfying the open set condition (see Section~\ref{sec: Iterated Function Systems} for more details). Let $\Kcal_i$ denote the limit set of $\Phi_i$, and set
\[
s_i := \dim_H(\Kcal_i).
\]
Define
\[
\Kcal := \Kcal_1 \times \cdots \times \Kcal_l \subset \R^m.
\]

    Fix $ x\in \widetilde{\X}$ and $0< q \leq 1$. Then the following bounds hold:
    \begin{align*}
        \dim_H(\Divergent(x,q,a,b)\cap \Kcal) \leq \min_{1 \leq k \leq l} \left(\frac{1}{w_k}\sum_i s_i \left( \max\{w_i, w_k\}- w_i q + w_i \EMassl(x) \right) \right) ,\\
        \dim_P(\Divergent(x,q,a,b) \cap \Kcal) \leq \min_{1 \leq k \leq l} \left( \frac{1}{w_k}\sum_i s_i \left( \max\{w_i, w_k\}- w_i q + w_i \EMassu(x) \right) \right).
    \end{align*}
\end{thm}
\begin{proof}
    Fix $x \in \widetilde{\X}$, $q \in(0,1]$ and $t>1$. Fix $\e>0$ and $\delta>0$ such that $\delta< \e c_i t^{-w_i}/2$ for all $i$. For notational simplicity, we will denote $g_t$ by $g$ throughout the proof. Let us define 
    \begin{align*}
        I(t, \e)&= \{k \in \N: \lambda_0(\pi(g_t^kx)) \leq 2\e\} , \\
        I(t, \e, N) &= I(t, \e) \cap [1, N],\\
          \oEM(x,\e, t) &= \limsup_{N \rightarrow \infty} \frac{1}{N} \# I(t, \e, N), \\
        \uEM(x,\e,t) &= \liminf_{N \rightarrow \infty} \frac{1}{N} \# I(t, \e, N).
    \end{align*}
    Using Lemma~\ref{lem: Emass}, it is clear that 
    \begin{align*}
         \lim_{\e \rightarrow 0} \oEM(x,\e, t) &= \EMassu(x), \\
         \lim_{\e \rightarrow 0} \uEM(x,\e, t) &= \EMassl(x).
    \end{align*}

    Let us now briefly recall some notation related to the fractal \(\Kcal\). For \(1 \leq i \leq l\), the set \(\Kcal_{i}\) is the limit set of the IFS \(\Phi_{i} = \{\phi_{i,e} : e \in E_{i}\}\), with a common contraction ratio \(c_{i}\) and cardinality \(p_{i} = \#E_{i}\). The dimension of \(\Kcal_{i}\) is given by \(s_{i} = -\frac{\log p_{i}}{\log c_{i}}\). The set \(\Kcal = \prod_{i} \Kcal_{i} \subset \mathbb{R}^m\) has dimension \(s = \sum_{i} s_{i}\). Let \(\mu_{i}\) denote the normalized restriction of \(\Hcal^{s_{i}}\) to \(\Kcal_{i}\), and define the measure \(\mu\) on \(\Kcal\) as \(\mu = \otimes_{i} \mu_{i}\). Also, let \(\alpha_i\) denote the diameter of \(\Kcal_i\) for \(1 \leq i \leq l\). Clearly, \(\alpha_i > 0\) if the dimension of \(\Kcal_i\) is not zero.

Note that if \(\dim_H(\Kcal_i) = 0\) for all \(i\), then the theorem holds trivially. Hence, we may assume that \(\dim_H(\Kcal_i) \neq 0\) for some \(i\). Let \(S\) denote the set of all \(1 \leq i \leq l\) such that \(\dim_H(\Kcal_i) = 0\) and set \(S^c = \{1, \ldots, l\} \setminus S\). For all \(i \in S^c\), let \(L_i\) be as defined in Lemma \ref{lem: finite intersect} and define \(L = \prod_i L_i\), where we set \(L_i = 1\) for \(i \in S\).

For all \(i \in S^c\), let us also define \(P_i(j) \in \mathbb{N}\) as the unique integer satisfying
\[
\alpha_i c_i^{P_i(j)+1} < 2 \delta t^{-jw_i} \leq \alpha_i c_i^{P_i(j)},
\]
and set \(P_i(j) = 1\) for \(i \in S\). Finally, for \(j \geq 1\), we define \(\Fcal(j) = \prod_{i} \Fcal_{i}(P_i(j))\).

     Fix $0< q'< q$. Note that
    \begin{align}
    \label{eq: Div q x}
        \Divergent(x, q, a, b) \cap \Kcal &\subset \bigcup_{M \in \N} Z(M),
    \end{align}
   where $Z(M)$ equals set of all $\xi \in \Kcal$ such that for all $N \geq M$, we have
    $$
    \frac{1}{N} \#\{k \in [1, N] \cap \N: \widetilde{\lambda}_0(g^k[I_d, v(\xi)]x) <\delta\} > q'.
    $$
    Also note that for all \(N > M\), we have  
\[
Z(M) \subset \bigcup_Q Z(M, N, Q),
\]
where the union is taken over all subsets \(Q \subset \{1, \ldots, N\}\) satisfying \(\#Q > q'N\). Here \(Z(M, N, Q)\) denotes the set of all \(\xi \in Z(M)\) such that, for every \(1 \leq k \leq N\),  
\[
\widetilde{\lambda}_0(g^k [I_d, v(\xi)] x) < \delta 
\quad \text{if and only if} \quad k \in Q.
\]
Before proceeding further, let us make some easy observations: \\

    \noindent {\bf Observation 1:} Let $1 < M < N$ and $Q \subset \{1, \ldots, N\}$ be such that $\#Q > q'N$. Assume that $1 \leq j \leq N$ and $R \in \mathcal{F}(j-1)$ are such that $R \cap Z(M, N, Q) \neq \emptyset$ and $j \in Q \setminus I(t, \epsilon)$. Then $R \cap Z(M, N, Q)$ is contained in a set of the form $B_1 \times \cdots \times B_l$, where each $B_i$ is a ball of radius $2\delta t^{-jw_i} $.

\noindent {\bf Explanation:} For $1 \leq i \leq l$, let 
\[
\begin{aligned}
    \rho_i &: \mathbb{R}^d = \mathbb{R}^{r_1} \times \cdots \times \mathbb{R}^{r_l} \times \mathbb{R}^n \to \mathbb{R}^{r_i}, \\
    \rho_0 &: \mathbb{R}^d = \mathbb{R}^m \times \mathbb{R}^n \to \mathbb{R}^n, \\
    \rho_i' &: \mathbb{R}^m = \mathbb{R}^{r_1} \times \cdots \times \mathbb{R}^{r_l} \to \mathbb{R}^{r_i},
\end{aligned}
\]
denote the natural projection maps. Suppose $\xi_1, \xi_2 \in R \cap Z(M, N, Q)$ are arbitrary. Since $j \in Q$, both affine lattices $g^j[I_d, v(\xi_1)]x$ and $g^j[I_d, v(\xi_2)]x$ contain a vector of size less than $\delta$, say $(v_1 + g^j v(\xi_1))$ and $(v_2 + g^j v(\xi_2))$, respectively. This implies that for all $1 \leq i \leq l$, we have
\[
\begin{aligned}
    \|\rho_i(v_1 - v_2)\| &\leq \|\rho_i(v_1 + g^j v(\xi_1))\| + \|\rho_i(v_2 + g^j v(\xi_2))\| + \|\rho_i(g^j v(\xi_1) - g^j v(\xi_2))\| \\
    &\leq 2\delta + t^{jw_i} \operatorname{diam}(R) \\
    &\leq 2\delta + t^{jw_i} \alpha_i c_i^{P_i(j-1)} \\
    &\leq 2\delta + t^{jw_i} c_i^{-1} 2\delta t^{-(j-1)w_i} \\
    &\leq (2 + t^{w_i} c_i^{-1} 2) \epsilon c_i t^{-w_i}/2 \\
    &\leq \epsilon + \epsilon = 2\epsilon,
\end{aligned}
\]
and 
\[
\begin{aligned}
    \|\rho_0(v_1 - v_2)\| &\leq \|\rho_0(v_1 + g^j v(\xi_1))\| + \|\rho_0(v_2 + g^j v(\xi_2))\| \\
    &\leq 2\delta \leq 2\epsilon.
\end{aligned}
\]
Therefore, $\|v_1 - v_2\| \leq 2\epsilon$. Note that $v_1 - v_2$ is an element of $\pi(g^j x)$, and since $j \notin I(t, \epsilon)$, we have $\lambda_0(\pi(g^j x)) > 2\epsilon$. This implies $v_1 - v_2$ must be the zero vector.

Thus, we have
\[
\begin{aligned}
    \|\rho_i'(\xi_1 - \xi_2)\| &\leq t^{-jw_i} \|\rho_i(g^j v(\xi_1) - g^j v(\xi_2))\| \\
    &\leq t^{-jw_i} \|\rho_i(v_1 + g^j v(\xi_1)) - \rho_i(v_2 + g^j v(\xi_2))\| \\
    &\leq t^{-jw_i} \left( \|\rho_i(v_1 + g^j v(\xi_1))\| + \|\rho_i(v_2 + g^jv(\xi_2))\| \right) \\
    &\leq 2\delta t^{- jw_i} .
\end{aligned}
\]
Hence, the observation follows. \\

    \noindent {\bf Observation 2:} Fix $1<M< N$ and $Q \subset \{1, \ldots, N\}$ satisfying $\#Q > q'N$. Then for all $ 1< j \leq N$, we have
    \begin{align}
        \label{eq:abb 1}
        \sum_{\substack{R \in \Fcal(j) \\ R \cap Z(M,N,Q) \neq \emptyset}} \mu(R) \leq \begin{cases}
             L \left(\prod_i {c_i^{s_i (P_i(j) - P_i(j-1))}}  \right) \sum_{\substack{R \in \Fcal(j-1) \\ R \cap Z(M,N,Q) \neq \emptyset}} \mu(R)  \quad \text{if $j \in Q \setminus I(t,\e)$}, \\
             \sum_{\substack{R \in \Fcal(j-1) \\ R \cap Z(M,N,Q) \neq \emptyset}} \mu(R) \quad \text{otherwise.}
        \end{cases}
    \end{align}
    \noindent {\bf Explanation:} First assume that $j \in Q \setminus I(t, \e)$. Fix $R \in \Fcal(j-1)$ such that $R \cap Z(M,N,Q) \neq \emptyset$. Then by Observation 1, we have $R \cap Z(M,N,Q)$ is contained in a set of form $B_1 \times \cdots \times B_l$, where each $B_i$ is a ball of radius $2 \delta  t^{-jw_j}$. By definition of $L$ and $\Fcal(j)$, it is clear that there are at most $L$-many elements in $\Fcal(j)$ which intersect $B_1 \times \cdots \times B_l$. Thus, we have that number of $R' \in \Fcal(j)$ such that $R' \subset R$ and $R \cap Z(M,N,Q) \neq \emptyset$ is at most $L$. Since $\mu(R)= \prod_i c_i^{s_i P_i(j-1)}$ and $\mu(R')= \prod_i c_i^{s_i P_i(j)}$ for any $R' \in \Fcal(j)$, we get that 
    \begin{align}
        \label{eq: abb 2}
        \sum_{\substack{R' \in \Fcal(j) \\ R' \subset R,\  R' \cap Z(M,N,Q) \neq \emptyset}}\mu(R') \leq L \left(\prod_i {c_i^{s_i P_i(j)}} \right) \leq L \left(\prod_i {c_i^{s_i (P_i(j)- P_i(j-1))}} \right) \mu(R)
    \end{align}
    The first case of \eqref{eq:abb 1} now follows from \eqref{eq: abb 2}. The second case of \eqref{eq:abb 1} is trivial. \\

     \noindent {\bf Observation 3:} Fix $1<M< N$ and $Q \subset \{1, \ldots, N\}$ satisfying $\#Q > q'N$. Then 
     \begin{align}
    \sum_{\substack{R \in \Fcal(N) \\ R \cap Z(M,N,Q) \neq \emptyset}} \mu(R) \leq L^N \left(\prod p_i \right)^N \left( \prod_i t^{- s_i w_i} \right)^{q'N - \#I(t,\e,N) }. \label{eq: abb 4} 
    \end{align}
    \noindent {\bf Explanation:} Note that by iteratively use of \eqref{eq:abb 1}, we have
    \begin{align}
    \label{eq: abb 3}
        \sum_{\substack{R \in \Fcal(N) \\ R \cap Z(M,N,Q) \neq \emptyset}} \mu(R) &\leq L^N \left(\prod_{j \in Q \setminus I(t,\e)}\prod_i {c_i^{s_i (P_i(j) - P_i(j-1))}}  \right). 
    \end{align}
    Also note that by definition of $P_i(j)$, for $ i  \in S^c $, we have
    \begin{align*}
      c_i^{P_i(j)} \leq \frac{2\delta}{\alpha_i c_i} t^{-jw_i}, \quad \text{ and } \quad  c_i^{-P_i(j-1)} \leq \frac{\alpha_i}{2\delta} t^{(j-1)w_i},
    \end{align*}
    and $P_{i}(j)= P_{i}(j-1)=1$ for $ i \in S$. Plugging this into \eqref{eq: abb 3} gives that 
    \begin{align*}
    \sum_{\substack{R \in \Fcal(N) \\ R \cap Z(M,N,Q) \neq \emptyset}} \mu(R) &\leq L^N \left(\prod_{j \in Q \setminus I(t,\e)}\prod_i t^{- s_i w_i} c_i^{-s_i}  \right) \nonumber \\
    &\leq L^N \left(\prod_i c_i^{-s_i} \right)^N \left( \prod_i t^{- s_i w_i} \right)^{\#(Q \setminus I(t,\e) )} \nonumber \\
    &\leq L^N \left(\prod_i p_i \right)^N \left( \prod_i t^{- s_i w_i} \right)^{q'N - \#I(t,\e,N) }. 
    \end{align*}
    Thus, the observation follows. \\

     \noindent {\bf Observation 4:} For all $1< M \leq N$, we have
     \begin{align*}
        \sum_{\substack{R \in \Fcal(N) \\ R \cap Z(M) \neq \emptyset}} \mu(R) \leq 2^N L^N \left(\prod p_i \right)^N \left( \prod_i t^{- s_i w_i} \right)^{q'N - \#I(t,\e,N) }.
    \end{align*}

    \noindent {\bf Explanation.} Note that \(Z(M)\) is the union of the sets \(Z(M, N, Q)\) with \(Q \subset \{1, \ldots, N\}\) and \(\#Q > q'N\). Since there are at most \(2^N\) possible choices for \(Q\), the observation follows directly from~\eqref{eq: abb 4}.\\

\noindent {\bf Observation 5:} Fix $1 \leq k \leq l$. For all $\gamma > 0$, define $N_\gamma \in \mathbb{N}$ as the unique integer satisfying $2 \delta t^{-N_\gamma w_k} \leq \gamma < 2\delta t^{-(N_\gamma-1)w_k}.$ Then for all $M > 1$ and sufficiently small $\gamma$, we have:
\begin{align}
    \label{eq: abb 6}
    \frac{\log ( C_\gamma(Z(M)))}{-\log (\gamma)} \leq 
    \frac{\log(D) + {N_\gamma} \log(B) + \sum_{i} \Big[s_i N_\gamma \max\{w_i, w_k\} - s_i w_i q'N_\gamma + s_i w_i \#I(t,\epsilon, N_\gamma)\Big] \log (t)}{
    -\log(2\delta) + (N_\gamma-1)w_k \log (t)},
\end{align}
where $C_\gamma(Z(M))$ denotes the smallest number of sets of diameter at most $\gamma$ that cover $Z(M)$, $D= \left( \prod_{i \in S^c} \left( \frac{\alpha_i }{2 \delta c_{i}} \right)^{s_i} \right)$ and $B= 2 L \left(\prod p_i \right)$. 

\noindent {\bf Explanation:}  
Fix $M > 1$ and $\gamma > 0$. Assume that $\gamma$ is small enough so that $N_\gamma > M$. For $i \in S^c$ and $j \in \mathbb{N}$, define $K_i(j)$ as the unique integer satisfying 
\begin{align*}
    \alpha_i c_i^{K_i(j)} < 2 \delta t^{-j \max\{w_k, w_i\}} \leq \alpha_i c_i^{K_i(j)-1},
\end{align*}
and set $K_i(j) = 1$ for $i \in S$. Clearly then for all $R \in \prod_{i} \mathcal{F}_{i}(K_i(N_\gamma))$, the diameter of $R$ is smaller than $\gamma$. Also note that $K_i(j) \geq P_i(j)$ for all $i \geq k$.  

To cover $Z(M)$ by sets of diameter less than or equal to $\gamma$, we select sets from $\prod_{i} \mathcal{F}_{i}(K_i(N_\gamma))$ that intersect $Z(M)$. The total number of elements in $\prod_{i} \mathcal{F}_{i}(K_i(N_\gamma))$ is:
\begin{align*}
    \prod_{i} p_{i}^{K_i(N_\gamma)} = \prod_{i} c_{i}^{-s_{i} K_i(N_\gamma)} \leq \prod_{i \in S^c} \left(\frac{\alpha_i}{2 \delta c_{i}} t^{N_\gamma \max\{w_k, w_i\}}\right)^{s_i}.
\end{align*}

Each element has an equal $\mu$-measure. Therefore, the number of sets covering $Z(M)$ satisfies:
\begin{align*}
    &\#\{R \in \prod_{i} \mathcal{F}_{i}(K_i(N_\gamma)) : R \cap Z(M) \neq \emptyset\} \\
    &\leq \prod_{i \in S^c} \left(   \frac{\alpha_i }{2 \delta c_{i}} t^{N_\gamma \max\{w_k,w_i\}} \right)^{s_i}. \left(\sum_{\substack{R \in \Fcal({N_\gamma}) \\ R\cap Z(M) \neq \emptyset }} \mu(R) \right) \\
        &\leq \prod_{i \in S^c} \left(   \frac{\alpha_i }{2 \delta c_{i}} t^{N_\gamma \max\{w_k,w_i\}} \right)^{s_i} . 2^{N_\gamma} L^{N_\gamma} \left(\prod p_i \right)^{N_\gamma} \left( \prod_i t^{- s_i w_i} \right)^{q'N_\gamma - \#I(t,\e,N_\gamma) } \\
        &\leq D B^{N_\gamma}   \prod_i   t^{s_iN_\gamma \max\{w_i, w_k\}      -s_iw_i q'N_\gamma + s_i w_i  \#I(t,\e,N_\gamma)} .
    \end{align*}
Thus, $Z(M)$ can be covered by at most:
\begin{align*}
    D B^{N_\gamma} \prod_i t^{s_i N_\gamma \max\{w_i, w_k\} - s_i w_i q'N_\gamma + s_i w_i \#I(t,\epsilon, N_\gamma)}
\end{align*}
sets of diameter at most $\gamma$. Since $\gamma < 2\delta t^{-(N_\gamma-1)w_k}$, the observation follows. \\

    Note that $N_\gamma \rightarrow \infty$ as $\gamma \rightarrow 0$. Therefore on taking $\liminf$ and $\limsup$ as $\gamma \rightarrow 0$ in \eqref{eq: abb 6}, we get from Lemma \ref{lem: Falconer} that
    \begin{align}
        \dim_P(Z(M)) &\leq  \frac{\log B + \sum_{i} \left(  s_i \max\{w_i, w_k\} -s_iw_i q' + s_i w_i \oEM(x, \e,t)   \right)\log(t) }{w_k \log t}, \label{eq: abb 7} \\
        \dim_H(Z(M)) &\leq \frac{\log B + \sum_{i} \left(  s_i \max\{w_i, w_k\} -s_iw_i q' + s_i w_i \uEM(x, \e,t)   \right)\log(t) }{w_k \log t}. \label{eq: abb 8}
    \end{align}
    Note that $\dim_P(\cup_i J_i) = \sup_i \dim_P(J_i)$ and $\dim_H(\cup_i J_i) = \sup_i \dim_H(J_i)$ for any countable collection of Borel sets $J_i$. Thus, from \eqref{eq: Div q x} and \eqref{eq: abb 7}, \eqref{eq: abb 8}, we get that
    \begin{align}
         \dim_H(\Divergent(x,q,a,b)) \leq \frac{\log B + \sum_{i} \left(  s_i \max\{w_i, w_k\} -s_iw_i q' + s_i w_i \uEM(x, \e,t)   \right)\log(t) }{w_k \log t}, \label{eq: abb 9}\\
        \dim_P(\Divergent(x,q,a,b)) \leq \frac{\log B + \sum_{i} \left(  s_i \max\{w_i, w_k\} -s_iw_i q' + s_i w_i \oEM(x, \e,t)   \right)\log(t) }{w_k \log t}. \label{eq: abb 10}
    \end{align}
    Since $B$ is independent of $q'$, $t$ and $\e$, first take limit as $\e \rightarrow 0$ in \eqref{eq: abb 9} and \eqref{eq: abb 10}, and then take limit as $t \rightarrow \infty$ and $q' \rightarrow q$ to get that
    \begin{align*}
         \dim_H(\Divergent(x,q,a,b)) &\leq \frac{1}{w_k}\sum_i s_i \left( \max\{w_i, w_k\}- w_i q + w_i \EMassl(x) \right), \\
          \dim_P(\Divergent(x,q,a,b)) &\leq \frac{1}{w_k} \sum_i s_i \left( \max\{w_i, w_k\}- w_i q + w_i \EMassu(x) \right).
    \end{align*}
    Since $1 \leq k \leq l$ is arbitrary, the theorem is proved.

\end{proof}

\section{Final Proof I}
\label{sec: Final Proof I}
\begin{proof}[Proof of Theorem \ref{main thm 2}]
    Note that in notation of Theorem \ref{gen thm 2}, we have 
    \begin{align*}
        \EMassl(\theta)&= \EMassl([u_\theta,0] \Z^d), \\
        \EMassu(\theta) &= \EMassu([u_\theta,0] \Z^d), \\
        \Div_\theta(a,b,q) &= \Divergent([u_\theta,0] \Z^d, a, b, q),
    \end{align*}
    for all $0< q \leq 1$. Thus, the theorem follows directly from Theorem \ref{gen thm 2}.
\end{proof}

\begin{proof}[Proof of Corollary~\ref{main thm 2'}]
    The theorem follows directly from theorem \ref{main thm 2} by choosing $l=m$, $r_1= \cdots = r_l= 1$, $\Kcal_i= [0,1]$ for all $i$.
\end{proof}

\begin{proof}[Proof of Corollary~\ref{main thm 2''}]
The first part of the corollary follows directly from Corollary~\ref{main thm 2'} by setting 
$a_1 = \cdots = a_m = 1/m$ and $b_1 = \cdots = b_n = 1/n$.  The second part then follows from the first together with \eqref{eq: inclusion}.
\end{proof}

\begin{proof}[Proof of Corollary~\ref{cor 1}]
Suppose $\EMassu(\theta)=0$. Then, using Corollary~\ref{main thm 2'}, \eqref{eq: inclusion}, and Lemma~\ref{lem: Falconer}, we have
\begin{align*}
    0 &\leq \dim_H(\Sing_\theta(a,b)) \leq \dim_P(\Sing_\theta(a,b)) \leq \dim_P(\Div_\theta(a,b,1)) \\
    &\leq \min_{1 \leq k \leq m} \frac{1}{a_k} \sum_{i=1}^m \Big( \max\{a_i, a_k\} - a_i \cdot 1 + a_i \EMassu(\theta) \Big) \\
    &\leq \frac{1}{a_m} \sum_{i=1}^m \Big( \max\{a_i, a_m\} - a_i + a_i \EMassu(\theta) \Big) \\
    &= \frac{1}{a_m} \sum_{i=1}^m \Big( a_i - a_i + a_i \cdot 0 \Big) = 0.
\end{align*}
Hence, the corollary follows.
\end{proof}

\begin{proof}[Proof of Corollary \ref{cor 2}]
    Note that if $\theta \notin \Div^0(a,b,1)$, then $\EMassl(\theta)<1$. Thus, we have
    \begin{align*}
        &\dim_H(\Sing_\theta(a,b)) \leq \dim_H(\Div_\theta(a,b,1)) \quad \text{using Lemma \ref{lem: Sing Dynamical Interpretation}} \\
        &\leq \min_{1 \leq k \leq m} \frac{1}{a_k}\sum_{i=1}^m \left( \max\{a_i, a_k\}- a_i  + a_i \EMassl(\theta) \right) \quad \text{ using Corollary~\ref{main thm 2'}} \\
        &\leq \frac{1}{a_1}\sum_{i=1}^m \left( \max\{a_i, a_1\}- a_i  + a_i \EMassl(\theta) \right)\\
        &= \sum_{i=1}^m \left( 1- \frac{a_i}{a_1} (1- \EMassl(\theta) ) \right) \\
        &< \sum_{i=1}^m 1 = m.
    \end{align*}
    The corollary now follows.
\end{proof}

\begin{proof}[Proof of Corollaries~\ref{cor intro 1} and~\ref{cor intro 3}]
The result follows directly from Corollary~\ref{cor 1}, together with the discussion in Remark~\ref{rem: wide measures}.
\end{proof}


\section{Notation II}
\label{sec: Notation II}
The following notation will be used for the rest of the paper.

\subsection{Iterated Function Systems}
For the rest of the paper, we fix for all $1 \leq i \leq m$ and $1 \leq j \leq n$, an iterated function system (IFS) $\Phi_{ij}=\{\phi_{ij,e}: e\in E_{ij}\}$ consisting of contracting similarities on $\R$ with equal contraction ratios, satisfying the open set condition. Let $p_{ij}= \#E_{ij}$ and let $c_{ij}$ denote the common contraction ratio of elements of $\Phi_{ij}$. Assume that the limit set of $\Phi_{ij}$, denoted by $\Kcal_{ij}$ has positive Hausdorff dimension, that is, $ \dim_H(\Kcal_{ij})= s_{ij}= -\log p_{ij}/ \log c_{ij} >0$.

Let us define $\Kcal =\{\theta \in \Mat : \theta_{ij} \in \Kcal_{ij} \} $ and $s= \sum_{ij} s_{ij}$. Let $\mu_{ij}$ denote the normalised restriction of $\Hcal^{s_{ij}}$ to $\Kcal_{ij}$ and define the measure $\mu= \otimes_{ij} \mu_{ij}$ on $\Kcal$. 

Let $\Xi \subset \Mat$ be defined as $\Xi=\{r \in \Mat: r_{ij} \in [c_{ij},c_{ij}^{-1}]\}$. For all $1 \leq i \leq m$, $1 \leq j \leq n$ and $r \in \Xi$, we define $\mur_{ij}$ as the measure on $\R$ obtained by pushing forward the measure $\mu_{ij}$ under map $ x \mapsto r_{ij}x$. We also define $\mur = \prod_{ij}\mur_{ij}$, viewed as measure on $\Mat$.

\subsection{Representation Theory}
\label{subsec: Rep Theory}
For all $1 \leq l \leq d$, define 
$$V_l=\bigwedge^l \R^{d}, \qquad   V= \bigoplus_{l=1}^{d} V_l .$$
Define action of $G$ on $V$ (resp. $V_l$) via the map $g \mapsto \bigoplus_{l=1}^{d} \bigwedge^l g$ (resp. $g \mapsto \wedge^l g$). Suppose $\{\bfe_1, \ldots , \bfe_{d}\}$ denote the standard basis of $\R^{d}$. For each index set $I = \{ i_1 < \cdots < i_l \} \subset \{1,\dots,d\}$, we define
  \begin{align*} 
  	{\bfe}_I := {\bfe}_{i_1} \wedge \cdots \wedge {\bfe}_{i_l}.
  \end{align*}
  The collection of monomials $\bfe_I$ with $\#I =l$, gives a basis of $V_l = \bigwedge^l \R^{d}$ for each $ 1\leq l\leq d$.
  For $v\in V$ and each index set $I$, we denote by $v_I \in \R$, the unique value so that $v= \sum_{J}v_J\bfe_J$, where the sum is taken over all index sets $J$. We define {\em norm} $\|.\|$ on each of $V$ as
  \begin{align}
  \label{eq: def ||||}
      \|v\|= \max_{I} |v_I|,
  \end{align}
  where the maximum is taken over all index sets $I$. For $g \in G$, we define
    \begin{align*}
        \|g\|:= \sup \left\{ \|gv\|: v\in V, \|v\| = 1 \right\}.
    \end{align*}
    Also, for any compact subset $Q \subset G$, we define
    \begin{align*}
          \|Q\| = \sup  \{ \|g\|, \|g^{-1}\|: g \in Q \},
     \end{align*}

For $1 \leq l \leq d$, we define $V_l^+$ to be the subspace of $V_l$ spanned by $\bfe_I$, where $I$ varies over the index sets satisfying $\#(I \cap \{1, \ldots, m\}) = \min\{l,m\}$. Similarly, define $V_l^-$ to be the subspace of $V_l$ spanned by $\bfe_I$, where $I$ varies over the index sets satisfying $\#(I \cap \{1, \ldots, m\}) \neq \min\{l,m\}$ Also define $\pi_{l+} $ (resp. $\pi_{l-}$) as the natural projection map from $V_l$ onto $V_l^+$ (resp. $V_l^-$). Note that for all $\theta \in \Mat$, we have $u(\theta)$ act trivially on $V_l^+$, i.e., $u(\theta)|_{V_{l}^+} = \mathrm{Id}_{V_{l}^+}$. We also define for $1 \leq l \leq d-1$, $w_l$ as least $w>0$ such that the subspace $V_{l,w}^+= \{v \in V_{l}^+: g_tv = t^wv\}$ is non-empty. It is easy to see that
\begin{align*}
    w_l &= \begin{cases}
         a_m + \cdots +a_{m-l+1} &\text{ if } l \leq m, \\
        b_n+ \cdots + b_{l-m+1}  &\text{ if } l \geq m.
    \end{cases}
\end{align*}

\subsection{Covolume of Lattice}
\label{subsec: Covolume of Lattice}
For a discrete subgroup $\Lambda$ of $\R^{d}$ of rank $l \geq 1$, we define $v_{\Lambda} \in V_l/\{\pm 1\}$ as $v_1 \wedge \cdots \wedge v_l$, where $v_1, \ldots, v_l$ is a $\Z$-basis of $\Lambda$. Note that the definition of $v_{\Lambda}$ is independent of the choice of basis $v_1, \ldots, v_l$. We define $\|\Lambda\|$ as
        \begin{align}
            \label{eq: def lattice covol}
            \|\Lambda\| = \|v_{\Lambda}\|,
        \end{align}
        where $\|.\|$ on $V_l$ is defined as in \eqref{eq: def ||||}. We also define $\|\{0\}\|=1$.

For $\Lambda \in \X$, let $P(\Lambda)$ denote the set of all \emph{primitive} subgroups of the lattice $\Lambda$, that is, the subgroups $L \subset \Lambda$ satisfying
\[
L = \Lambda \cap \sspan_\R(L),
\]
where $\sspan_\R(L)$ denotes the smallest real vector subspace of $\R^d$ containing $L$.

We will need the following important lemma.
  \begin{lem}[{\cite[Lem.~5.6]{EMM98}}]
    \label{lem EMM}
    There exists a constant $D > 0$ such that the following inequality holds. For all $\Lambda \in \X$ and for all $\Lambda_1, \Lambda_2 \in P(\Lambda)$, we have:
    \begin{align}
        \label{eq: EMM98}
        \|\Lambda_1 \cap \Lambda_2\| \|\Lambda_1 + \Lambda_2\| \leq D \|\Lambda_1\| \|\Lambda_2\|.
    \end{align}
\end{lem}

\begin{rem}
    In \cite{EMM98}, inequality \eqref{eq: EMM98} is established with $D = 1$, but the norm $\|\Lambda\|$ is defined differently. There, $\|\Lambda\|$ is taken as $\|v_{\Lambda}\|$, where $\|.\|$ on $V_l = \wedge^l \R^d$ is the norm induced by the Euclidean norm on $\R^d$. Since any two norms on a finite-dimensional vector space are equivalent, it follows that \eqref{eq: EMM98} holds for some sufficiently large $D$ under our current definition of $\|\Lambda\|$.
\end{rem}

\section{Dimension Bound in Generalized Setup II}
\label{sec: Dimension Bound in Generalized Setup II}
The following section is taken from \cite{aggarwalghoshsingular}.

\begin{defn} [The Contraction Hypothesis]
      \label{def: Contraction Hypothesis}
      Suppose $Y$ is a metric space equipped with an action of $G$. Given a collection of functions $\{ f_\tau: Y \rightarrow (0,\infty]: \tau \in S\}$ for some unbounded set $S \subset (0,\infty)$ and $\beta > 0$, we say that $\mu$ satisfies the $((f_\tau)_\tau, \beta)$-\textbf{contraction hypothesis} on $Y$ if the following properties hold: 

\begin{enumerate}
    \item The set $Y_f = \{y \in Y : f_\tau(y) = \infty\}$ is independent of $\tau$ and is $G$-invariant.

    \item For every $\tau \in S$, $f_\tau$ is uniformly log-Lipschitz with respect to the $G$-action. That is, for every bounded neighborhood $\mc{O}$ of the identity in $G$, there exists a constant $C_\mc{O} \geq 1$ such that for all $g \in \mc{O}$, $y \in Y$, and $\tau \in S$,
    \begin{align*}
        C_\mc{O}^{-1} f_\tau(y) \leq f_\tau(g y) \leq C_\mc{O} f_\tau(y).
    \end{align*}

    \item There exists a constant $c \geq 1$ such that the following holds: for every $\tau \in S$, there exists $T > 0$ such that for all $y \in Y$, $r \in \Xi$, and $f_\tau(y) > T$,
    \begin{align*}
        \int_{\Mat} f_\tau(g_\tau u(x) y) \, d\mur(x) \leq c f_\tau(y) \tau^{-\beta}.
    \end{align*}
\end{enumerate}
The functions $f_\tau$ will be referred to as \textbf{height functions}.
\end{defn}

\begin{defn}
\label{def:div}
   Suppose $Y$ is a locally compact second countable metric space equipped with a continuous $G$ action. Given a closed $G$-invariant subset $Y' \subset Y$, $0 < p \leq 1$ and $y \in Y \setminus Y'$, we define $\Diver(y,Y', p)$ as set of all $x \in \Mat$ such that
   $$
    \liminf_{T \rightarrow \infty } \frac{1}{T} \int_0^T \delta_{g_{e^t}u(x)y}(Y \setminus K) \, dt \geq p,
   $$
   for all compact subsets $K \subset Y \setminus Y'$.
\end{defn}

We have the following theorem.
\begin{thm}[{\cite[Thm.~6.5]{aggarwalghoshsingular}}]
\label{thm: contraction implies dimesnion bound}
    	Let $Y$ be a locally compact second countable metric space equipped with a continuous action of $G$. Assume that there exists a collection of functions $\{ f_\tau: Y \rightarrow (0,\infty]: \tau \in S\}$ for some unbounded set $S \subset (0,\infty)$ and $0<\beta< (a_1+b_1)s $, such that $\mu$ satisfies the $(\{f_\tau\}_{\tau \in S}, \beta)$-contraction hypothesis on $Y$. Assume that $Y_f = \{y \in Y : f_\tau(y) = \infty\}$, which is independent of $\tau$ and is $G$-invariant. Then for all $y\in Y \backslash Y_f$ and $0<p\leq 1$, 
        \begin{align*}
        \dim_P \left( \Diver(y,Y_f, p) \cap \Kcal \right) \leq s-\frac{p\beta}{a_1+b_1}.
        \end{align*}
        Also, for any sequence $(c_{\tau})_{\tau \in S}$ of positive real numbers and $0<a \leq (a_1+ b_1)s- \beta$, we have
        \begin{align*}
        \dim_P \left(x\in \mc{K}: \substack{\text{ for all $\tau \in S$, the following holds for all sufficiently large $t$} \\  f_\tau(g_tu(x)y) \geq c_\tau t^{a} }  \right) \leq s-\frac{a+\beta }{a_1+b_1} .
        \end{align*}
\end{thm}


\section{Height Function}
\label{sec: Height Function}
The section is devoted towards the construction of a family of height functions $\{f_\tau\}$ on $\widetilde{\X}$ such that the set $\{f_\tau = \infty\}$ equals $\X \subset \widetilde{\X}$.

\begin{defn}
    For each $1 \leq l \leq d-1$, we define the $l$-th critical exponent $\zeta_l(\mu)$ of the measure $\mu$ as the supremum of all $\gamma \geq 0$ for which there exists a constant $C_{\gamma,i}' > 0$ such that, for every $v = v_1 \wedge \cdots \wedge v_l \in V_l$ with $\|v\| = 1$ and $r \in \Xi$, the following inequality holds:
\begin{align*}
   \int_{\Mat} \frac{1}{\|\pi_{l+}(u(\theta)v)\|^{\gamma}} \, d\mur(\theta) < C_{\gamma,l}'.
\end{align*}
\end{defn}

We will need the following result from \cite{aggarwalghoshsingular}. 
\begin{prop}[{\cite[Prop.~3.1, Lemma~4.1, 4.4 and 4.5]{aggarwalghoshsingular}}]
\label{Critical Exponent is positive}
    For all $1 \leq l \leq d-1$, we have $\zeta_l(\mu) >0$. Moreover, the critical exponent $\zeta_l(\mu)$ satisfies the following lower bound in the following special cases:
    \begin{itemize}
        \item If $\Kcal= \M_{m \times n}([0,1])$, then 
   $$
   \zeta_l(\mu) \geq  \begin{cases}
       \frac{m}{l} \quad \text{if } l \leq m, \\
       \frac{n}{m+n-l} \quad \text{if } m < l \leq d-1.
   \end{cases}
   $$
   \item  If $n=1$, then
   $$\zeta_l(\mu)\geq \min \left\{\sum_{i \in I}s_{i1}: \#I= d-l \right\} $$  

   \item If $m=1$, then
   $$\zeta_l(\mu)\geq \min \left\{\sum_{i \in I}s_{1i}: \#I= l \right\}.$$
    \end{itemize}
    
\end{prop}

For the remainder of this section, we fix a sequence $\eta_1, \ldots, \eta_{d-1} \in \mathbb{R}$ such that the following holds:
\begin{align*}
    0 < \eta_i &< \zeta_i(\mu), \quad \text{for } 1 \leq i \leq d-1, \\
    \frac{1}{\eta_{i-j}} + \frac{1}{\eta_{i+j}} &< \frac{2}{\eta_i}, \quad \text{for all } 1 \leq i \leq d-1 \text{ and } j \leq \min\{i, d-i\},
\end{align*}
where we define $\frac{1}{\eta_0} = \frac{1}{\eta_d} := 0$. Additionally, we define the following:
\begin{align*}
    \eta &= \min_{1 \leq l \leq d} w_l \eta_l, \\
    \bfn &= (\eta, \eta_1, \ldots, \eta_{d-1}), \\
    C_{\bfn} &= \max_{1 \leq l \leq d-1} C_{\eta_l, l}'.
\end{align*}

The constants chosen above satisfy the following. 
\begin{prop}[{\cite[Prop.~5.1]{aggarwalghoshsingular}}]
    \label{prop: Critical Exponent representation}
    For all $1 \leq l \leq d-1$, $r \in \Xi$, $t>1$ and $v= v_1 \wedge \cdots \wedge v_l \in V_l \setminus \{0\}$, the following holds
    \begin{align*}
         \int_{\Mat}  \|g_{t} u(x)v\|^{- \eta_l} \, d\mur(x) \leq C_{\bfn} t^{-\eta } \|v\|^{-\eta_l }.
    \end{align*}
\end{prop}

For every $0 \leq l \leq d$, we define $\varphi_{l}: \X \rightarrow \R$ as
    \begin{align*}
    \varphi_{l}(\Lambda) =  \max \{\|\Lambda_l\|^{-1}: \Lambda_l \in P(\Lambda), \mathrm{ rank}(\Lambda_l)=l  \},
    \end{align*}
    where $\|.\|$ is defined as in \eqref{eq: def lattice covol}. Also, for $0 \leq l \leq d$, define $\tilde{\varphi}_l: \widetilde{\X} \rightarrow (0,\infty)$ as $\tilde{\varphi}_l = \varphi_l \circ \pi$. Then it is easy to see that $\Tilde{\varphi}_0 \equiv \Tilde{\varphi}_{d} \equiv 1$.  We also define $\psi: \widetilde{\X} \rightarrow (0, \infty]$ as
    \begin{align*}
        \psi(\widetilde{\Lambda}) = \max_{v \in \widetilde{\Lambda}} \|v\|^{-1} = \max \{\|v\|^{-1}: v \in \R^d, \widetilde{\Lambda }= [I_{d},v]\pi(\widetilde{\Lambda})\},
    \end{align*}
    for all $\widetilde{\Lambda} \in \widetilde{\X}$. Note that $\psi(\widetilde{\Lambda}) = \infty$ if and only if $\widetilde{\Lambda} \in \X$.

\begin{prop}
    \label{prop: Contraction psi}
    For all $t >1$, there exists $\xi(t) \geq 1$, such that the following holds for all $r \in \Xi$ and $\widetilde{\Lambda} \in \widetilde{\X}$,
    $$
    \int_{\Mat} \psi^{\eta_1}(g_tu(x)\widetilde{\Lambda}) \, d\mur(x ) \leq C_{\bfn} t^{- \eta } \psi^{\eta_1}(\widetilde{\Lambda}) + \xi(t) \tilde{\varphi}_1^{\eta_1}(\widetilde{\Lambda}).
    $$
\end{prop}
\begin{proof}
    Fix $t>1$ and $\widetilde{\Lambda} \in \widetilde{\X}$. Let $\xi'(t)= 2\|\{g_{t} u(x): x \in  \bigcup_{r \in \Xi} \supp(\mur)\}\|$. Let $v_0 \in \widetilde{\Lambda}$ be a vector in $\widetilde{\Lambda}$ such that $\psi(\widetilde{\Lambda})= \|v_0\|^{-1}$. Claim that for all $x \in  \bigcup_{r \in \Xi} \supp(\mur)$
    \begin{align}
        \label{eq: a b c 1}
        \psi( g_tu(x)\widetilde{\Lambda}) \leq \max \left\{ \frac{1}{\|g_tu(x)v_0\|} , \  \xi'(t) \tilde{\varphi}_1(\widetilde{\Lambda}) \right\}.
    \end{align}
   To see this claim, note that for any \(x \in \bigcup_{r \in \Xi} \supp(\mu_r)\), if 
\[
\psi(g_t u(x)\widetilde{\Lambda}) > \xi'(t)\, \tilde{\varphi}_1(\widetilde{\Lambda}),
\]
then there exists a vector \(v_x \in \widetilde{\Lambda}\) such that 
\[
\|g_t u(x) v_x\|^{-1} = \psi(g_t u(x)\widetilde{\Lambda}) > \xi'(t)\, \tilde{\varphi}_1(\widetilde{\Lambda}).
\]
This implies that
    \begin{align}
    \label{eq: abcdef 1}
        \|v_x\| \leq \|(g_tu(x))^{-1}\|. \|g_t u(x) v_x\| <  \frac{\xi'(t)}{2} \frac{1}{\xi'(t) \tilde{\varphi}_1(\widetilde{\Lambda}) } = \frac{1}{2 \tilde{\varphi}_1(\widetilde{\Lambda})}.
    \end{align}
    Thus for all $w \in \widetilde{\Lambda} \setminus \{v_x\}$, we have
    \begin{align*}
        \|w\| &\geq \|w-v_x\| - \|v_x\| \\
        &\geq \min\{ \|w'\|: w' \in \pi(\widetilde{\Lambda}) \setminus \{0\} \} - \|v_x\|, \quad \text{ since  $w-v_x \in \pi(\widetilde{\Lambda})$ and $w \neq v_x$} \\
        &\geq \frac{1}{ \tilde{\varphi}_1(\widetilde{\Lambda})} - \frac{1}{2 \tilde{\varphi}_1(\widetilde{\Lambda})} \quad \text{ using \eqref{eq: abcdef 1}}\\
        &= \frac{1}{2 \tilde{\varphi}_1(\widetilde{\Lambda})} \\
        &> \|v_x\| \quad \text{ using \eqref{eq: abcdef 1}}.
    \end{align*}
    This means that $v_x$ is the shortest vector of $\widetilde{\Lambda}$, hence must equal $v_0$. Thus $$\psi(g_tu(x)\widetilde{\Lambda})= \|g_tu(x) v_0\|^{-1}.$$ This proves the claim.

    Let $\xi(t)= (\xi'(t))^{\eta_1}$. Using \eqref{eq: a b c 1}, we have
    \begin{align*}
        \int_{\Mat} \psi^{\eta_1}(g_tu(x)\widetilde{\Lambda}) \, d\mur(x )  &\leq \int_{\Mat} \frac{1}{\|g_tu(x)v_0\|^{\eta_1}} \, d\mur(x) + \xi(t) \tilde{\varphi}_1^{\eta_1}(\widetilde{\Lambda}) \\
        &\leq C_{\bfn} t^{-\eta} \frac{1}{\|v_0\|^{\eta_1}} + \xi(t) \tilde{\varphi}_1^{\eta_1}(\widetilde{\Lambda}) \\
        &= C_{\bfn} t^{- \eta } \psi^{\eta_1}(\widetilde{\Lambda}) + \xi(t) \tilde{\varphi}_1^{\eta_1}(\widetilde{\Lambda}),
    \end{align*}
     where penultimate inequality follows from Proposition \ref{prop: Critical Exponent representation}. Hence, the proposition follows. 
\end{proof}

\begin{prop}
        \label{prop: contraction varphi}
         For all $t >1$, there exists $\xi(t) \geq 1$, such that the following holds for all $1 \leq l \leq d-1$, $r \in \Xi$ and $\widetilde{\Lambda} \in \widetilde{\X}$ 
        \begin{align*}
            \int_{\Mat} \tilde{\varphi}_{l}^{ \eta_l  }(g_{t} u(x) \widetilde{\Lambda}) \, d\mur({x})
            &\leq C_{\bfn} t^{- \eta } \tilde{\varphi}_l^{ \eta_l }(\widetilde{\Lambda}) + \xi(t) \left( \max_{1 \leq j \leq \min\{l, d-l\}}   \tilde{\varphi}_{l-j}(\widetilde{\Lambda})  \tilde{\varphi}_{l+j}(\widetilde{\Lambda}) \right)^{ \eta_l/2 }.
        \end{align*}
    \end{prop}
\begin{proof}
    Using the fact that $\tilde{\varphi}_{l}^{ \eta_l  }(g_{t} u(x) \widetilde{\Lambda})= \varphi_l(g_tu(x) \Lambda)$, where $\Lambda = \pi(\widetilde{\Lambda})$, the proposition follows immediately from \cite[Prop.~5.2]{aggarwalghoshsingular}.
\end{proof}

Let us define
$$
\an = \min\left\{ 1- \frac{\eta_i}{2}\left(  \frac{1}{\eta_{i-j}}+ \frac{1}{\eta_{i+j}} \right): 1 \leq i \leq d-1, 1 \leq j \leq \min\{i, d-i\} \right\},
$$
where we set $1/\eta_0 = 1/\eta_d = 0$. For $0<\e<1$, we define the function $f_{\e, \bfn}: \widetilde{\X} \rightarrow \R$ as 
    \begin{align*}
    	f_{\epsilon, \bfn}(\widetilde{\Lambda}) = \e^{-2} + \e^{-1} \left(\sum_{l=1}^{d-1}   \tilde{\varphi}_{l}^{\eta_l }(\widetilde{\Lambda}) \right) +  \psi^{\eta_1}(\widetilde{\Lambda}).
    \end{align*}

The definition of $f_{\epsilon,\bfn}$ is motivated from \cite[Section~5]{Shi20}.

\begin{prop}\label{thm: height function}
    For all $t > 1$, there exists $b= b(t, \bfn) \geq 0$ and $0< \e = \e(t,\bfn )< 1$ such that the following holds for all $\widetilde{\Lambda} \in \widetilde{\X}$ and $ r \in \Xi$ 
    \begin{align}
        \label{eq: height fn contraction}
        \int_{\Mat} f_{\e,\bfn}(g_tu(x)\widetilde{\Lambda}) \, d\mur(x) \leq 3C_{\bfn} t^{-\eta} f_{\e, \bfn}(\Lambda) + b.
    \end{align}
\end{prop}
\begin{proof}
    Fix $t>1$. Let $\xi(t)$ be the maximum of the constants provided by Propositions \ref{prop: Contraction psi} and \ref{prop: contraction varphi}. Let $0<\e<1$ be a constant to be determined. Suppose $\Tilde{\Lambda} \in \widetilde{\X}$ be arbitrary. Then using Propositions \ref{prop: Contraction psi} and \ref{prop: contraction varphi}, we get that 
    \begin{align}
        &\int_{\Mat} f_{\e,\bfn}(g_tu(x) \widetilde{\Lambda}) \, d\mur(x) \nonumber \\
        &= \e^{-2} +  \e^{-1} \left(\int_{\Mat} \sum_{l=1}^{d-1}  \tilde{\varphi}_l^{\eta_l }(g_tu(x)\widetilde{\Lambda}) \, d\mur(x) \right) + \int_{\Mat} \psi^{\eta_1}(g_tu(x) \widetilde{\Lambda}) \, d\mur(x) \nonumber \\ 
        &\leq \e^{-2}  + C_{\bfn} t^{-\eta} \e^{-1} \left( \sum_{l=1}^{d-1} \tilde{\varphi}_l^{\eta_l }(\widetilde{\Lambda}) \right) + C_{\bfn} t^{-\eta} \psi^{\eta_1}(\widetilde{\Lambda})  \nonumber \\
        &+ \e^{-1} \xi(t) \left( \sum_{l=1}^{d-1}  \max_{1 \leq j \leq \min\{l, d-l\}} \left(  \tilde{\varphi}_{l-j}(\widetilde{\Lambda})\tilde{\varphi}_{l+j}(\widetilde{\Lambda})\right)^{\eta_l/2} \right) +  \xi(t) \tilde{\varphi}_1^{\eta_1}(\widetilde{\Lambda}). \label{eq: w 9}
    \end{align}
    Note that
    \begin{align}
        \label{eq: w 9 9 1}
        C_{\bfn} t^{-\eta} \e^{-1} \left( \sum_{l=1}^{d-1} \tilde{\varphi}_l^{\eta_l }(\widetilde{\Lambda}) \right) + C_{\bfn} t^{-\eta} \psi^{\eta_1}(\widetilde{\Lambda}) = C_{\bfn}t^{-\eta} \left( f_{\e, \bfn}(\widetilde{\Lambda}) -\e^{-2} \right),
    \end{align}
    and
    \begin{align}
        \label{eq: w 9 9}
        \tilde{\varphi}_1^{\eta_1}(\widetilde{\Lambda}) \leq \e f_{\e, \bfn}(\widetilde{\Lambda}).
    \end{align}
    Also, we have
    \begin{align}
        \tilde{\varphi}_{l-j}(\widetilde{\Lambda}) &\leq  (\e f_{\e, \bfn}(\widetilde{\Lambda})) ^{\frac{1}{\eta_{l-j}}}, \label{eq: 1 1 1} \\
        \tilde{\varphi}_{l+j}(\widetilde{\Lambda}) &\leq  (\e f_{\e, \bfn}(\widetilde{\Lambda}))^{\frac{1}{\eta_{l+j}}}, \label{eq: 1 1 2} \\
        1 &\leq \left(\e^{2} f_{\e, \bfn}(\widetilde{\Lambda})) \right)^{1-\frac{\eta_l}{2}\left(\frac{1}{\eta_{l-j}} + \frac{1}{\eta_{l+j} } \right)} \label{eq: 1 1 3}.
    \end{align}
    Thus
    \begin{align}
        \left(  \tilde{\varphi}_{l-j}(\widetilde{\Lambda}) \tilde{\varphi}_{l+j}(\widetilde{\Lambda})\right)^{\eta_l/2} &\leq \left( (\e f_{\e, \bfn}(\widetilde{\Lambda})) ^{\frac{1}{\eta_{l-j}}}  (\e f_{\e, \bfn}(\widetilde{\Lambda}))^{\frac{1}{\eta_{l+j}}} \right)^{\eta_l/2} \quad \text{ using \eqref{eq: 1 1 1}, \eqref{eq: 1 1 2}}  \nonumber \\
        &= (\e f_{\e, \bfn}(\widetilde{\Lambda}))^{\frac{\eta_l}{2}\left(\frac{1}{\eta_{l-j}} + \frac{1}{\eta_{l+j} } \right)} \nonumber\\
        &\leq (\e f_{\e, \bfn}(\widetilde{\Lambda}))^{\frac{\eta_l}{2}\left(\frac{1}{\eta_{l-j}} + \frac{1}{\eta_{l+j} } \right)} \left(\e^{2} f_{\e, \bfn}(\widetilde{\Lambda})) \right)^{1-\frac{\eta_l}{2}\left(\frac{1}{\eta_{l-j}} + \frac{1}{\eta_{l+j} } \right)} \nonumber     \quad \text{ using \eqref{eq: 1 1 3}}\\
        &\leq \e^{1+\an} f_{\e, \bfn}(\widetilde{\Lambda}). \label{eq: w 10}
    \end{align}
    Thus, we get from \eqref{eq: w 9}, \eqref{eq: w 9 9 1}, \eqref{eq: w 9 9} and \eqref{eq: w 10} that 
    \begin{align}
        \int_{\Mat} f_{\e,\bfn}(g_tu(x)\widetilde{\Lambda}) \, d\mur(x)  &\leq C_{\bfn} t^{-\eta} f_{\e,\bfn}(\widetilde{\Lambda}) + \e^{-2}(1- C_{\bfn}t^{-\eta})  \nonumber \\
        &+ (d-1) \e^{\an} \xi(t) f_{\e, \bfn}(\widetilde{\Lambda}) + \e \xi(t) f_{\e, \bfn}(\widetilde{\Lambda}). \label{eq: 1 1 1 1}
    \end{align}
    Choose $\e$ small enough so that $(d-1) \e^{\an} \xi(t) \leq {C_{\bfn} t^{-\eta}}$ and $\e \xi(t) \leq C_{\bfn} t^{-\eta} $. Also choose $b= \max\{0, \e^{-2}(1- C_{\bfn}t^{-\eta})\}$. Then for this choice of $\e, b$, we get \eqref{eq: height fn contraction} follows from \eqref{eq: 1 1 1 1}. This proves the proposition.

    \end{proof}


\section{Final Proof II}
\label{sec: Final Proof II}

\begin{prop}
    \label{prop:Genral Estimate}
    Let $\eta, \eta_1, \ldots, \eta_{d-1} \in \R$ be a sequence satisfying the following conditions:
    \begin{align*}
        0<\eta_i &\leq \zeta_i(\mu) \quad \text{for all } 1 \leq i \leq d-1,  \\
        \frac{1}{\eta_{i-j}} + \frac{1}{\eta_{i+j}} &\leq \frac{2}{\eta_i} \quad \text{for all } 1 \leq i \leq d-1, \, 1\leq j \leq \min\{i, d-i\},  \\
        \eta &= \min_{1 \leq l \leq d} w_l \eta_l,
    \end{align*}
    where $1/\eta_0 = 1/\eta_d := 0$. 

   Then, the following bounds hold for all $0<\gamma \leq (s(a_1+ b_1)- \eta)/\eta_1$, $0< p \leq 1$ and $x \in \widetilde{\X} \setminus \X$
    \begin{align*} 
        \dim_P( \Diver(x,\X,p) \cap \Kcal) &\leq s - \frac{p\eta}{a_1 + b_1}, \\ 
        \dim_P( \{ \theta \in \Kcal: \substack{ \text{ there exists $T_\theta >0$ such that for all $t>T_\theta$ ,} \\ \text{ we have $\psi( g_tu(\theta) x) \geq t^{\gamma} $} } \}) &\leq s - \frac{1}{a_1 + b_1} \left( \eta + \eta_1 \gamma \right),
    \end{align*}
    where $\Diver(\cdot)$ is defined as in Definition~\ref{def:div}.
\end{prop}

\begin{proof}
We divide the proof into two cases. \\
{\bf Case 1}
    In this case, we assume $\eta, \eta_1, \ldots, \eta_{d-1}$ satisfies following strict inequalities
        \begin{align*}
            \eta_i &< \zeta_i(\mu) \quad \text{ for all } 1 \leq i \leq d-1 \\
           \frac{1}{\eta_{i-j}} + \frac{1}{\eta_{i+j}} &< \frac{2}{\eta_i} \text{ for all } 1 \leq i \leq d-1, 1 \leq j \leq \min\{ i, d-i\}.
        \end{align*}
    In this case, using Proposition \ref{thm: height function}, for every $t>1$, choose $\e(t)$ and define the collection of height functions
$$
 \{f_t:= f_{\e(t), \bfn}: t >1\}.
$$
Now it is easy to see that the action of \(G\) on \(\widetilde{\X}\) satisfies the \(((f_t)_{t \geq 1}, \eta)\)-contraction hypothesis with respect to the measure \(\mu\). To see this, note that
\[
f_t = \e(t)^{-2} + \e(t)^{-1} \left( \sum_{l=1}^{d-1} \tilde{\varphi}_{l}^{\eta_l}(\widetilde{\Lambda}) \right) + \psi^{\eta_1}(\widetilde{\Lambda}),
\]
for all $t>1$. 

Since \(\tilde{\varphi}_l(\widetilde{\Lambda}) < \infty\) for all \(\widetilde{\Lambda} \in \widetilde{\X}\) and \(\psi(\widetilde{\Lambda}) = \infty\) if and only if \(\widetilde{\Lambda} \in \X\), it follows that for all \(t > 1\), the set \(\{ f_t = \infty \}\) equals \(\X\). This verifies the first property of Definition~\ref{def: Contraction Hypothesis}.  

For the second property, note that each of the functions \(\psi, \tilde{\varphi}_1, \ldots, \tilde{\varphi}_{d-1}\), and the constant function are log-Lipschitz with respect to the \(G\)-action. Hence, their linear combinations (in particular, the family \((f_t)_{t > 1}\)) are uniformly log-Lipschitz with respect to the \(G\)-action. This establishes the second property of Definition~\ref{def: Contraction Hypothesis}.  

The third property follows from Proposition~\ref{thm: height function}, taking \(c = 4C_{\bfn}\) and \(T = b t^{\eta}/C_{\bfn}\) corresponding to each \(t\) (note that the value of \(b\) also depends on \(t\)).

Thus, by Theorem \ref{thm: contraction implies dimesnion bound} and the fact that
$$
\{ \theta \in \Kcal: \substack{ \text{ there exists $T_\theta >0$ such that for all $t>T_\theta$ ,} \\ \text{ we have $\psi( g_tu(\theta) x) \geq t^{\gamma} $} } \} \subset \{ \theta \in \Kcal: \substack{\text{ there exists $T_\theta >0$ such that for all $t>T_\theta$ and $\tau >1$,} \\ \text{ we have $f_\tau( g_tu(\theta) x) \geq t^{\eta_1 \gamma} $ }} \},
$$
the proposition follows in this case.

{\bf Case 2} In this case, fix \(\eta,\ \eta_1,\dots,\eta_{d-1}\) which satisfy the hypotheses of the proposition but do not lie in Case~1.

To proceed, set \(q_i := i(d - i)\) for \(0 \le i \le d\), and for every \(\delta > 0\), define the sequences
\[
\eta_j^{(\delta)}:=\frac{1}{\frac{1}{\eta_j}+\delta q_j}\quad(1\le j\le d-1),
\qquad
\eta^{(\delta)}:=\min_{1\le l\le d-1} w_l\,\eta_l^{(\delta)}.
\]
Then \(\eta^{(\delta)},\eta_1^{(\delta)},\dots,\eta_{d-1}^{(\delta)}\) lie in Case~1. Indeed, \(0<\eta_1^{(\delta)}<\eta_1\le\zeta_i(\mu)\), and for every admissible \(i,j\),
\begin{align*}
\frac{2}{\eta_i^{(\delta)}}-\Big(\frac{1}{\eta_{i-j}^{(\delta)}}+\frac{1}{\eta_{i+j}^{(\delta)}}\Big)
&=2\Big(\frac{1}{\eta_i}+\delta q_i\Big)-\Big(\frac{1}{\eta_{i-j}}+\delta q_{i-j}+\frac{1}{\eta_{i+j}}+\delta q_{i+j}\Big)\\
&=\Big(\frac{2}{\eta_i}-\frac{1}{\eta_{i-j}}-\frac{1}{\eta_{i+j}}\Big)
+ \delta \Big(2q_i-(q_{i-j}+q_{i+j})\Big)\\
&\ge 0 + \delta \big(2i(d-i)-(i-j)(d-i+j)-(i+j)(d-i-j)\big)\\
&=2 \delta j^2>0,
\end{align*}
where we used that \((\eta_i)_i\) satisfies the conditions of the proposition. Hence Case~I applies and yields
\begin{align*}
\dim_P\big(\Diver(x,\X,p)\cap\Kcal\big)
&\le s-\frac{p\,\eta^{(\delta)}}{a_1+b_1},\\[4pt]
\dim_P\Big\{\theta\in\Kcal:\ \exists T_\theta>0\text{ s.t. }\forall t>T_\theta,\ \psi\big(g_tu(\theta)x\big)\ge t^\gamma\Big\}
&\le s-\frac{1}{a_1+b_1}\Big(\eta^{(\delta)}+\frac{\eta_1^{(\delta)}a_mb_n\omega}{a_m+b_n+a_m\omega}\Big).
\end{align*}
Finally, letting \(\delta\to0\) (so that \(\eta^{(\delta)}\to\eta\) and \(\eta_j^{(\delta)}\to\eta_j\)), we obtain the asserted bounds in this case as well.
\end{proof}

\begin{proof}[Proof of Theorem~\ref{main thm 1}]
Fix $\xi \in \R^m$. Using Remark \ref{rem: Enough for irrational}, we may assume that $\xi \notin \Z^m$. Then the element $x= [I_d, v(\xi)] \Z^d $ does not belong to $\X$ and we have
\begin{align}
    \label{eq: temp w 1} 
    \Div^\xi(a,b,p) &\subset \Diver(x, \X, p), \\
     \label{eq: temp w 2}
     \Sing^\xi(a,b,\omega) &\subset \bigcap_{\omega'< \omega} \{\theta \in \Mat: \text{ for all large $t$, we have } \psi(g_tu(\theta)x)\geq t^{\frac{ a_mb_n\omega'}{a_m+b_n+ a_m\omega'}} \},
\end{align}
where \eqref{eq: temp w 2} follows from Lemma \ref{lem: omega sing dynamical Interpretation}.

Also note that by Proposition \ref{Critical Exponent is positive}, we know that $\zeta_l(\mu) > 0$ for all $1 \leq l \leq d-1$. Therefore, we can construct a sequence $\eta_1, \ldots, \eta_{d-1}$ such that:
    \begin{align}
        0<\eta_i &\leq \zeta_i(\mu), \label{eq: 1 2 1} \\
        \frac{1}{\eta_{i-j}} + \frac{1}{\eta_{i+j}} &\leq \frac{2}{\eta_i} \quad \text{for all } 1 \leq i \leq d-1, \, j \leq \min\{i, d-i\}, \label{eq: 1 2 2}
    \end{align}
    where $1/\eta_0 = 1/\eta_d := 0$. For any such sequence, the results in \eqref{eq: main thm 1} and \eqref{eq: main thm 2} follow directly from \eqref{eq: temp w 1}, \eqref{eq: temp w 2} and Proposition \ref{prop:Genral Estimate}. This completes the proof of the first part of the theorem.

    For the second part, observe that Proposition~\ref{Critical Exponent is positive}, together with the inequalities
\[
\frac{m}{l} \leq d-l \quad \text{and} \quad \frac{n}{d-l} \leq l, \qquad (1 \leq l \leq d-1),
\]
ensures that the constants defined in \eqref{eq: main thm 3}, \eqref{eq: main thm 4}, and \eqref{eq: main thm 5} satisfy the condition in \eqref{eq: 1 2 1}. Moreover, the constants defined in \eqref{eq: main thm 3}, \eqref{eq: main thm 4}, and \eqref{eq: main thm 5} also satisfy the condition in \eqref{eq: 1 2 2}. Hence, the theorem follows.

\end{proof}

\begin{proof}[Proof of Corollary~\ref{cor intro 2}]
The corollary follows directly from Theorem~\ref{main thm 1} and \eqref{eq: inclusion}. 
In particular, we choose $\Kcal_{ij} = [0,1]$ for all $i,j$ and set
\[
\eta_l =
\begin{cases}
    \dfrac{m}{l}, & \text{if } l \leq m, \\
    \dfrac{n}{m+n-l}, & \text{if } l > m,
\end{cases}
\]
in equation~\eqref{eq: main thm 1} and~\eqref{eq: main thm 2}.  
For these choice of $\eta_l$, we have
\[
\min_{1 \leq l \leq d-1} \eta_l w_l = \min\{m a_m, n b_n\},
\]
and hence the result follows \eqref{eq: inclusion},~\eqref{eq: main thm 1} and~\eqref{eq: main thm 2}.
\end{proof}

\begin{proof}[Proof of Corollary~\ref{cor intro 4}]
The corollary again follows directly from Theorem~\ref{main thm 1} and \eqref{eq: inclusion}. 
In particular, we choose $\Kcal_{ij} = \mathcal{C}_3$ for all $i,j$ and set
\[
\eta_l =
\begin{cases}
    \dfrac{m}{l} \dfrac{\log 2}{\log 3}, & \text{if } l \leq m, \\
    \dfrac{n}{m+n-l} \dfrac{\log 2}{\log 3}, & \text{if } l > m,
\end{cases}
\]
in equations~\eqref{eq: main thm 1} and~\eqref{eq: main thm 2}. 
Since either $m=1$ or $n=1$, this choice is admissible by equations~\eqref{eq: main thm 4} and~\eqref{eq: main thm 5}. 

For these choice of $\eta_l$, we have
\[
\min_{1 \leq l \leq d-1} \eta_l w_l = \min\{m a_m, n b_n\} \dfrac{\log 2}{\log 3},
\]
and hence the result follows from \eqref{eq: inclusion}, \eqref{eq: main thm 1}, and \eqref{eq: main thm 2}.
\end{proof}

\bibliography{Biblio}
\end{document}